\documentclass[11pt]{article}
\usepackage{fullpage}
\usepackage[utf8]{inputenc}
\usepackage{amsmath}
\usepackage{amssymb}
\usepackage{amsfonts}
\usepackage{bbm}
\usepackage{graphicx}
\usepackage{float}
\usepackage{subcaption}
\usepackage{smod}
\usepackage[ruled]{algorithm2e}
\usepackage{array}

\newcommand{\p}{\mathbb{P}}
\newcommand{\e}{\mathbb{E}}

\usepackage{verbatim}
\usepackage{listings}
\usepackage{color}

\newcommand{\var}{\ensuremath{\operatorname{Var}}}
\newcommand{\eye}{\ensuremath{\Ibb}}
\newcommand{\cov}{\ensuremath{\operatorname{Cov}}}
\newcommand{\T}{^\intercal}

\newcommand{\supp}{\operatorname{supp}}
\newcommand\independent{\protect\mathpalette{\protect\independenT}{\perp}}
\def\independenT#1#2{\mathrel{\rlap{$#1#2$}\mkern2mu{#1#2}}}
\usepackage{epsfig}
\usepackage{wasysym}
\newcommand{\bfX}{\mathbf{X}}
\newcommand{\bfZ}{\mathbf{Z}}
\newcommand{\bfw}{\mathbf{w}}

\def\smbbeta{ {\scriptsize\bbeta}}
\def\Ibb{\mathbb{I}}
\def\Xbb{\mathbb{X}}
\def\Dsc{\mathcal{D}}

\definecolor{darkred}{RGB}{150,50,50}
\definecolor{brown}{RGB}{150,000,000}
\definecolor{green}{RGB}{000,050,000}
\definecolor{purple}{RGB}{050,050,250}

\def\Jsc{\mathcal{J}}
\def\DKL{D_{\mbox{\tiny KL}}}

\usepackage[usenames,dvipsnames,svgnames,table]{xcolor}
\usepackage[colorlinks=true,
            linkcolor=red,
            urlcolor=blue,
            citecolor=blue]{hyperref}

\definecolor{lbcolor}{rgb}{0.95,0.95,0.95}
\begin{document}

\title{$L_1$-Regularized Least Squares for Support Recovery of High Dimensional Single Index Models with Gaussian Designs}

\author{Matey Neykov\thanks{Department of Operations Research and Financial Engineering, Princeton University, Princeton, NJ 08544} \and Jun S. Liu\thanks{Department of Statistics, Harvard University, Cambridge, MA 02138} \and Tianxi Cai \thanks{Department of Biostatistics, Harvard University, Boston, MA 02115}}

\date{}

\maketitle

\begin{abstract}  It is known that for a certain class of single index models (SIMs)  $Y = f(\bX_{p \times 1}\T\bbeta_0, \varepsilon)$, support recovery is impossible when $\bX \sim \cN(0, \Ibb_{p \times p})$ and a \textit{model complexity adjusted sample size} is below a critical threshold. Recently, optimal algorithms based on Sliced Inverse Regression (SIR) were suggested. These algorithms work provably under the assumption that the design $\bX$ comes from an i.i.d. Gaussian distribution. In the present paper we analyze algorithms based on covariance screening and least squares with $L_1$ penalization (i.e. LASSO) and demonstrate that they can also enjoy optimal (up to a scalar) rescaled sample size in terms of support recovery, albeit under slightly different assumptions on $f$ and $\varepsilon$ compared to the SIR based algorithms. Furthermore, we show more generally, that LASSO succeeds in recovering the signed support of $\bbeta_0$ if $\bX \sim \cN(0, \bSigma)$, and the covariance $\bSigma$ satisfies the irrepresentable condition. Our work extends existing results on the support recovery of LASSO for the linear model, to a more general class of SIMs.
\end{abstract}


\section{Introduction}

Modern data applications often require scientists to deal with high-dimensional problems in which the sample size $n$ could be much less than the dimensionality of the covariates $p$. To handle such challenging problems, structural assumptions on the data generating mechanism are often imposed. Such assumptions are motivated by the fact that classical procedures such as linear regression provably fail, unless the ratio $p/n$ converges to $0$. However, modern procedures based on regularization may work well under the high dimensional setting with additional sparsity assumptions. In the sparse high dimensional setting, it is often of interest to uncover the sparsity pattern, or in other words to select the relevant variables for the model. Under
generalized linear models, support recovery can be achieved by fitting these models with penalized optimization procedures such as LASSO \citep{tibshirani1996regression} or Dantzig Selector \citep{candes2007dantzig}, which are computationally inexpensive compared to exhaustive search approaches. The LASSO algorithm's variable selection/support recovery capabilities under generalized linear models, have been extensively studied \citep[e.g. among others]{meinshausen2006high, zhao2006model, wainwright2009sharp, lee2013model}. %\tcomm{add in ref for glm} 
%In the present paper we attempt to relax the assumption that the linear model is correctly specified in a certain sense.
However, much less is known under potential mis-specification of these commonly used models or under more general models. 

In this paper, we focus on recovering the support of the regression coefficients $\bbeta_0$ under a \textit{single index model} (SIM): 
\begin{equation} \label{SIM}
Y = f(\bX\T\bbeta_0, \varepsilon), \quad \varepsilon \independent \bX
\end{equation}
where both the link function $f$ and the distribution of $\varepsilon$ are left unspecified. Throughout,  we assume that $\bbeta_0\T \bSigma \bbeta_0 = 1$ for identifiability  and $E(\bX) = 0$, where $\bSigma = E(\bX\bX\T)$. We are specifically interested in the case where $\bX \in \RR^p \sim \cN(0, \bSigma)$ and $\bbeta_0$ is $s$-sparse with $s < p$. Obviously, the SIM includes many commonly used parametric or semi-parametric regression models such as the linear regression model as special cases. Under (\ref{SIM}) and the sparsity assumption, we aim to show that the standard least squares LASSO algorithm,
\begin{align} \label{LASSOproblem:intro}
\widehat \bbeta = \argmin_{\smbbeta \in \RR^p}  \left\{\frac{1}{2n} \sum_{i = 1}^n(Y_i - \bX_i\T \bbeta)^2 + \lambda \|\bbeta\|_1 \right\},
\end{align}
can successfully recover the support of $\bbeta_0$ provided standard regularity conditions and that the {\em model complexity adjusted effective sample size,} $$n_{p,s} = n/\{s\log(p-s)\},$$ is 
sufficiently large. Obviously, for most choices of $f$, fitting (\ref{LASSOproblem:intro}) is essentially making inference under the mis-specified linear regression model. 

The least squares LASSO algorithm has been used frequently in practice to perform variable selection for analyzing genomic data \citep[e.g.]{cantor2010prioritizing, zhao2011pathway, wang2013genome}. However, the linear model with Gaussian error is unlikely to be the true model in many such cases. Hence it is of practical importance to theoretically establish that the LASSO's support recovery capabilities are in fact robust to mis-specification. In addition to arguing that LASSO is robust, and perhaps even more surprisingly, we demonstrate that fitting the mis-specified linear model with LASSO penalty can optimally (up to a scalar) achieve support recovery with respect to the effective sample size $n_{p,s}$, for certain classes of $\bSigma$ and SIMs in a minimax sense. In the special case when $\bSigma =  \Ibb_{p \times p}$, the LASSO algorithm can be slightly modified into a simple covariance screening procedure which possesses similar properties as the LASSO procedure.

%that even if one is handed a misspecified linear regression model, the LASSO algorithm succeeds in support recovery with optimal sample size, and hence is robust to model misspecification modulo certain further assumptions on the function $f$ and error $\varepsilon$. 

\subsection{Overview of Related Work}

When the dimension $p$ is small, inference under a SIM has been studied extensively in the literature \cite[e.g.]{xia1999single, horowitz2009semiparametric, peng2011penalized, mccullagh1989generalized} among many others. In the highly relevant line of work on sufficient dimension reduction, many seminal insights can be found in \cite{li1989regression,li1991sliced, cook2005sufficient}. When $\bX \sim \cN(0, \bSigma)$, results given in \cite{li1989regression} can be used to show that $\mbox{argmin}_{\smbbeta}\{\sum_{i = 1}^n(Y_i - \bX_i\T \bbeta)^2\}$
consistently estimates $\bbeta_0$ up to a scalar. When $\bbeta_0$ is sparse, the sparse sliced inverse regression procedure given in \cite{li2006sparse} can be used to effectively recover $\bbeta_0$ under model (\ref{SIM}) although their procedure requires a consistent estimator of $\bSigma^{-1/2}$.

In the high dimensional setting with diverging $p$, \cite{alquier2013sparse} were the first to consider the sparse SIM, and proposed an estimation framework using a PAC-Bayesian approach. \cite{wang2015distribution} and \cite{wang2012non} demonstrated that when support recovery can be achieved when $p = O(n^k)$ under a SIM via optimizations in the form of
$\widehat \bbeta = \argmin_{\smbbeta \in \RR^p} \frac{1}{2n} \sum_{i = 1}^n(F_n(Y_i) - 1/2 - \bX_i\T \bbeta)^2 + \sum_{j = 1}^p J_\lambda(\beta_j),$
where $J_\lambda$ is a penalty function and $F_n(x) = \frac{1}{n}\sum_{i = 1}^n \mathds{1}(Y_i \leq x)$. 
Regularized procedures have also been proposed for specific choices of $f$ and $Y$. For example, \cite{yi2015optimal} study consistent estimation under the model $\PP(Y = 1 | \bX) = \{f(\bbeta\T \bX) + 1\}/2$ with binary $Y$, where $f: \RR \mapsto [-1,1]$. \cite{yang2015sparse} consider the model $Y = f(\bX\T\bbeta) + \varepsilon$ with known $f$, and develop estimation and inferential procedures based on the $L_1$ regularized least squares loss. 

With $p$ potentially growing with $n$ exponentially and under a general SIM, \cite{radchenko2015high} proposed a non-parametric least squares with an equality $L_1$ constraint to handle simultaneous estimation of $\bbeta$ as well as $f$. The support recovery properties of this procedure are not investigated, and in addition the results do not exhibit the optimal scaling of the triple $(n,p,s)$. \cite{han2015provable} suggest a penalized approach, in which they use a loss function related to Kendall's tau correlation coefficient. They also establish the $L_2$ consistency for the coefficient $\bbeta$ but do not consider support recovery. 
\cite{neykov2015signed} analyzed two algorithms based on Sliced Inverse Regression \citep{li1991sliced} under the assumption that $\bX \sim \cN(0, \eye_{p \times p})$, and demonstrated that they can uncover the support optimally in terms of the rescaled sample size.  
\cite{versh2015thegeneralized} and \cite{hassibi2015thelasso} demonstrated that a constrained version of LASSO can be used to obtain an $L_2$ consistent estimator of $\bbeta_0$. None of these procedures provide results on the performance of the LASSO algorithm in support recovery, which relates to $L_2$ consistency but is a fundamentally different theoretical aspect. In addition, no existing work on the SIM estimation procedures demonstrates that the performance depends on $(n,p,s)$ only through the effective sample size $n_{p,s}$. 

\subsection{Organization}

The rest of the paper is organized as follows. Our main results are formulated in section \ref{main:results:sec}. In particular, we show  results on the covariance screening algorithm when $\bSigma = \Ibb_{p \times p}$ in section \ref{covariancescreening}  and our main result on the LASSO support recovery in section \ref{cov:struct:section}. Proof for the main results are given in section \ref{proof:main:result}. In addition we demonstrate that for a class of SIMs, any algorithm provably fails to recover the support, unless the rescaled sample size $n_{p,s}$ is large enough. Numerical studies, confirming our main result are shown in section \ref{simulations:sec}. We discuss potential future directions in section \ref{discussion:sec}. Technical proofs are deferred to the appendixes.

\section{Main Results} \label{main:results:sec}

In this section we formulate our main results, which include the analysis of a simple covariance screening algorithm, and the LASSO algorithm for SIMs. Before we move on to the algorithms we summarize notation that we use throughout the paper and discuss several useful definitions and preliminary results.

\subsection{Preliminary and Notation}

For a (sparse) vector $\vb = (v_1, \ldots, v_p)\T$, we let $S(\vb) := \supp(\vb) = \{j : v_j \neq 0\}$ denote its support,  $S_{\pm}(\vb) := \{(\sign(v_j),j) : v_j \neq 0\}$ be its signed support, $\|\vb\|_p$ denote the $L_p$ norm, 
$\|\vb\|_0=|\supp(\vb)|$, and $\vb^{\min} = \min_{i \in \supp(\vb)} |v_i|$.
For a real random variable $X$, define
$$
\|X\|_{\psi_2} = \sup_{p \geq 1} p^{-1/2} (\e|X|^p)^{1/p}, ~~~ \|X\|_{\psi_1} = \sup_{p \geq 1} p^{-1} (\e|X|^p)^{1/p}.\footnotemark
$$
 \footnotetext{There are multiple equivalent (up to universal constants) definitions of the so-called Orlicz or $\psi$ norms. See \cite{vershynin2010introduction} Lemma 5.5 for a succinct formal treatment of this.}Recall that a random variable is called \textit{sub-Gaussian} if $\|X\|_{\psi_2} < \infty$  and \textit{sub-exponential} if $\|X\|_{\psi_1} < \infty$. For any integer $k \in \NN$ we use the shorthand notation $[k] = \{1,\ldots,k\}$. For a matrix $\Mb \in \mathbb{R}^{d_1 \times d_2}$, sets $S_1, S_2 \subseteq [d]$, we let $\Mb_{,S_2} = [M_{ij}]_{i\in[d_1]}^{j \in S_2}$ and $\Mb_{S_1, S_2} = [M_{ij}]_{i \in S_1}^{j \in S_2}$. For a vector $\bZ = (Z_1, \ldots, Z_d)\T$ and set $S \subseteq [d]$, $\bZ_{S}$ denotes the subvector corresponding to the set $S$. Furthermore, let $\|\Mb\|_{p,q} = \sup_{\|\vb\|_p = 1} \|\Mb \vb\|_q$. In particular, we have
$
\|\Mb\|_{2,2} = \max_{i \in [\max(d_1,d_2)]}\{s_i(\Mb)\},
$ 
where $s_i(\Mb)$ is the $i$\textsuperscript{th} singular value of $\Mb$, and $ \|\Mb\|_{\infty,\infty} = \max_{i \in [d_1]} \sum_{j =1}^{d_2} |M_{ij}|$. 
For a matrix $\Mb \in \RR^{d \times d}$, we put $D_{\max}(\Mb) = \max_{i \in [d]}|M_{ii}|$ for its maximal diagonal element, and $\diag(\Mb) = [M_{ii}]_{i \in [d]}$ for the collection of diagonal entries of $\Mb$.
%Let $F_n(x) = \frac{1}{n}\sum_{i = 1}^n \Ibb(Y_i \leq x)$ denote the empirical distribution of the $Y$ sample, and $\Phi$ ({\it resp.} $\phi$) denote the cdf ({\it resp.} pdf) of a standard normal random variable.
We also use standard asymptotic notations. Given two sequences $\{a_n\}, \{b_n\}$ we write $a_n = O(b_n)$ if there exists a constant $C < \infty$ such that $a_n \leq C b_n$; $a_n = \Omega(b_n)$ if there exists a positive constant $c> 0$ such that $a_n \geq c b_n$, $a_n = o(b_n)$ if $a_n/b_n \rightarrow 0$, and $a_n \asymp b_n$ if there exists positive constants $c$ and $C$ such that $c < a_n/b_n < C$. Throughout, we also assume that there exists a constant $0 < \iota < 1$ such that $s < p - p^{\iota}$, which implies that $\frac{\log(p-s)}{\log(p)}\geq \iota$.

We assume that data for analysis consists of $n$ independent and identically distributed (i.i.d.) random vectors $\Dsc=\{(Y_i, \bX_i\T)\T, i = 1, \ldots, n\}$ and we focus primarily on $\bX \sim \cN(0,\bSigma)$. In matrix form, we let $\Xbb = [\bX_1, \ldots, \bX_n]_{n \times p}\T = [X_{ij}]_{i \in [n]}^{j \in [p]}$, $\bY = (Y_1, \ldots, Y_n)\T$, and $\bvarepsilon = (\varepsilon_1, \ldots, \varepsilon_n)\T$. The recovery of $\bbeta_0$ under SIM often relies on the linearity of expectation assumption
given in \cite{li1989regression} and \cite{li1991sliced}: 
\begin{definition}[Linearity of Expectation] A $p$-dimensional random variable $\bX$ is said to satisfy linearity of expectation in the direction $\bbeta$ if for any direction $\bb \in \RR^p$:
$$
\e[\bX\T \bb | \bX\T \bbeta] = c_{\bb} \bX\T \bbeta + a_{\bb},
$$
where $a_{\bb}, c_{\bb} \in \RR$ are some real constants which might depend on the direction $\bb$.
\end{definition}
\begin{remark} Note that if additionally $\e[\bX] = 0$, then by taking expectation it is evident that $\ba_{\bb} \equiv 0$. Clearly, linearity of expectation is direction specific by definition. Elliptical distributions \citep{fang1990symmetric} including multiviariate normal are known to satisfy the linearity in expectation uniformly in all directions \cite[e.g.]{cambanis1981theory}. 
\end{remark}

Next, we record a simple but very useful observation, which forms the basis of our work. From Theorem 2.1 of \cite{li1989regression}, we note that under a SIM and normality of $\bX$ with covariance $\bSigma > 0$, 
$$\argmin_{\bb} \e(Y - \bb\T \bX)^2 = c_0 \bbeta_0,$$ 
for some $c_0 \in \mathbb{R}$. More generally, we have

\begin{lemma} \label{proplemma} Assume that the SIM (\ref{SIM}) holds, $\bSigma = \e(\bX\bX\T) >0$, and $\bX$ satisfies the linearity in expectation condition in the direction $\bbeta_0$ such that $\e[(\bX\T \bbeta_0)^2] > 0$.  
Then we have $\bSigma^{-1}\e(Y \bX) = c_0  \bbeta_0$, where $c_0 :=  \e(Y\bX\T\bbeta_0)$. Obviously, $\e(Y \bX) = c_0 \bbeta_0$ when $\bSigma = \Ibb_{p \times p}$.
\end{lemma}
%\begin{remark}\label{propgenremark} The statement of Lemma \ref{proplemma} is readily generalizable to the situation where $Y = f(\bX\T \bbeta, \varepsilon)$, but $\e(\bX\bX\T) = \bSigma > 0$ with $\bSigma \neq \eye$. This is equivalent to $Y = f(\bX\T\bSigma^{-1/2} \bSigma^{1/2} \bbeta, \varepsilon)$, and we conclude that $\bSigma^{-1/2}\e(Y \bX) = c_0 \bSigma^{1/2}\bbeta$, where $c_0 := \dfrac{\e[f(Z,\varepsilon)Z]}{\bbeta\T \bSigma \bbeta}$.
%\end{remark}

In view of  Lemma \ref{proplemma}, under sparsity assumptions, an $L_1$ regularized least square estimator can recover $\bbeta_0$ proportionally and hence the support of $\bbeta_0$.  Furthermore, when $\bSigma = \Ibb_{p \times p}$, the covariance $\e(Y\bX)$ can be directly used to recover $\bbeta_0$. It is noteworthy to remark that in the special case $\bX \sim \cN(0,\bSigma)$, a simple application of Stein's Lemma \citep{stein1981estimation} can help quantify the constant $c_0$ precisely:
$$
c_0 = \EE(Y \bX\T \bbeta_0) = \EE (Z \underbrace{[\EE f(Z, \varepsilon) | Z]}_{\varphi(Z)}) = \int D \varphi(z) \frac{\exp(-z^2/2)}{\sqrt{2\pi}} dz= \EE D \varphi(Z),
$$
where $D \varphi$ is the distributional derivative of $\varphi$, $Z \sim \cN(0,1)$ and we abused the notation slightly in the last equality for simplicity.

The remaining of this section is structured as follows --- in section \ref{covariancescreening} we study a simple covariance thresholding algorithm, which is a manifestation of program (\ref{LASSOproblem:intro}) under the assumption $\bSigma = \eye_{p \times p}$. In section \ref{cov:struct:section} we consider the full-fledged least squares LASSO algorithm (\ref{LASSOproblem:intro}) with a general covariance matrix $\bSigma$. Throughout, we assume $E(\bX) = 0$ and let $\sigma^2 = E(Y^2)$, $\eta = \Var(Y^2)$, 
$$
c_0 =  \e(Y\bX\T\bbeta_0), \ \gamma = \Var(Y\bX\T\bbeta_0) , \
\xi^2 = \e\{(Y - c_0 \bX\T\bbeta_0)^2\}, \ \mbox{and}\ \theta^2 =  \var\{(Y - c_0 \bX\T\bbeta_0)^2\} . %\label{shorthandnot2}
$$
In addition, to simplify the presentation we will assume that the above constants are not scaling with $(n,p,s)$, and belong to a compact set which is bounded away from $0$. 
%\tcomm{is it really just a matter of presentation? i thought you need these constants to be free of $(n,p,s)$}\\
%\mcomm{in principle we could always allow these to scale and trace them within the proofs. then these parameters }\\
%\mcomm{will modify the rates in the convergence in certain ways, but I think this will make the presentation pretty hard}\\\ \mcomm{to read. we can certainly remove this sentence in any case.}

\subsection[]{Covariance Screening under $\bSigma = \Ibb_{p \times p}$} \label{covariancescreening}

In this section, we propose a simple covariance screening procedure for signed support recovery of $\bbeta_0$ under a SIM with $\bSigma = \Ibb_{p \times p}$ , which relates to the sure independence screening procedures  \citep{fan2008sure, fan2010sure} under the linear model. Note that the LASSO procedure  (\ref{LASSOproblem:intro}) can equivalently be expressed as:
$$
\widehat \bbeta = \argmin_{\smbbeta \in \RR^p} \left\{ - \frac{1}{n} \sum_{i = 1}^n Y_i \bX_i\T \bbeta + \frac{1}{2n} \bbeta\T \sum_{i = 1}^n  \bX_i \bX_i\T \bbeta + \lambda \|\bbeta\|_1 \right\}.
$$
Under the assumption $\bSigma = \eye_{p \times p}$, we replace $\frac{1}{n}\sum_{i = 1}^n  \bX_i \bX_i\T$ with $\Ibb_{p \times p}$ and instead consider
$$
\widehat \bbeta = \argmin_{\smbbeta \in \RR^p} \left\{- \frac{1}{n} \sum_{i = 1}^n Y_i \bX_i\T \bbeta + \frac{1}{2} \bbeta\T \bbeta + \lambda \|\bbeta\|_1 \right\}.
$$
It is well known  that the solution to the above program takes the form:
$$
\hat \beta_j = \sign\Big(n^{-1}\sum_{i = 1}^n Y_i X_{ij}\Big)\Big(\Big|n^{-1}\sum_{i = 1}^n Y_i X_{ij}\Big| - \lambda\Big)_+,
$$
where $x_+=x \mathds{1}(x \ge 0)$. Hence, in this special case, the regularization parameter can be equivalently interpreted as a thresholding parameter, filtering all small $| n^{-1}\sum_{i = 1}^n Y_i X_{ij}|$. Motivated by this, we consider the following
%\begin{align}\label{cov:ave} \hat \cov(Y, \bX) := n^{-1}\sum_{i = 1}^n Y_i \bX_i, \end{align}
%which is simply a vector estimating $\cov(Y, \bX)$. In this section we will study a 
simple covariance screening procedure, which acts as a filter taking covariances and their corresponding signs, only if they pass a critical threshold: %\tcomm{specify how you choose $s$ instead of using true $s$}
%\begin{algorithm}[H] \label{Cov:screening}
%\SetKwInOut{Input}{input}
%\Input{$(Y_i, \bX_i)_{i = 1}^n$: data, $s$: the sparsity of $\bbeta$}
%\begin{enumerate}
%\item Calcluate $\Vb := \hat \cov(Y, \bX) := n^{-1}\sum_{i = 1}^n Y_i \bX_i$, %as in (\ref{cov:ave});
%\item Collect the $s$ highest $|V_j|$ into the set $\widehat S$;
%\item Output the set $\{(\sign(v_j), j): |V_j| \in \widehat S\}$.
%\end{enumerate}
%\caption{Covariance Screening Algorithm}
%\end{algorithm}

\begin{algorithm}[H] \label{Cov:screening}
\SetKwInOut{Input}{input}
\Input{$(Y_i, \bX_i)_{i = 1}^n$: data; tuning parameter $\nu > 0$}
\begin{enumerate}
\item Calcluate $\bV := \hat \cov(Y, \bX) := n^{-1}\sum_{i = 1}^n Y_i \bX_i$, %as in (\ref{cov:ave});
%\item Collect the $s$ highest $|V_j|$ into the set $\widehat S$;
\item Set $\widehat S := \Big\{V_j : |V_j| > \nu \sqrt{\frac{\log p}{n}}\Big\}$;
\item Output the set $\{(\sign(V_j), j): |V_j| \in \widehat S\}$.
\end{enumerate}
\caption{Covariance Screening Algorithm}
\end{algorithm}

%From Lemma \ref{proplemma} it follows that under the assumption $\e f(Z, \varepsilon) Z \neq 0$, where $Z \sim \bX\T \bbeta$, the proportionality constant $c_0 \neq 0$. Since for identifiability we are assuming that $\|\bbeta\|_2 = 1$, in fact we have $c_0 = \e f(Z, \varepsilon) Z \neq 0$. We remark that while in general, the assumption $\e f(Z, \varepsilon) Z \neq 0$ is dependent on the particular direction $\bbeta$, in the case when $\bX$ comes from a spherical distribution \citep{mardia1979multivariate}, such as the standard normal, the projections $\bX\T \ub$ where $\ub$ is any unit vector have precisely the same distribution. Recall that a random variable is spherically symmetric if it is elliptically distributed with $\bSigma = \eye_{p \times p}$. As we mentioned earlier, elliptical distributions and hence spherical distributions as a special case, automatically satisfy linearity in expectation. For the remaining of the section we will assume that $\bX$ has a spherical distribution and that in addition it is sub-Gaussian. 
%\begin{assumption}[Spherical Sub-Gaussian Distribution]\label{spherical:sub:gauss}
%Let $\bX$ with $\e[\bX] = 0, \var[\bX] = \Ibb_{p \times p}$ is a spherically distributed $p$ dimensional random variable, with characteristic function $\EE \exp(i \bt\T\Xb)\equiv \Psi(\bt\T \bt), \bt \in \RR^p$, and in addition we assume that $\sup_{\|\ub\|_2 = 1}\|\bX\T \ub\|_{\psi_2} \leq K$. %, where $\Psi : \RR \mapsto \RR$ is such that $\Psi(t) \leq \exp(-Ct)$ for some $C > 0$ for all $t \in \RR^+$. 
%\end{assumption}

The following proposition shows that the algorithm recovers the support with probability approaching 1 provided that the effective sample size $n_{p,s}$ is sufficiently large under the normality assumption of $\bX \sim \cN(0, \Ibb_{p \times p})$. %\tcomm{since $\log(p) \asymp \log(p-s)$, please just use $n_{p,s}$ instead of also using $n/\{s\log(p)\}$ whenever possible.}
\begin{proposition} \label{propnormal} 
Assume that $\bX \sim \cN(0, \Ibb_{p \times p})$, $c_0 \ne 0$, $E(Y^4) < \infty$, and $\beta_{0j} \in \{\pm\frac{1}{\sqrt{s}}, 0\}$ for all $j \in [p]$ and some $s \in \mathbb{N}$. Set $\nu = \frac{1}{\sqrt{s}} + 2 \sqrt{2} \sigma$. Then as long as 
$$n_{p,s} \geq \Upsilon,$$
for a large enough $\Upsilon = O(1)$ (depending on $c_0, \sigma^2$) Algorithm \ref{Cov:screening} recovers the signed support of $c_0\bbeta_0$ with asymptotic probability $1$.
\end{proposition}
%%
%\tcomm{move to appendix}

The proof of Proposition \ref{propnormal} can be found in Appendix \ref{cov:tresh:app}. We note that this result does not require $Y$ to be sub-Gaussian. When sub-Gaussianity is assumed, the convergence rate of the probability approaching 1 can be improved. 
Furthermore, under sub-Gaussianity of $Y$, we can relax the normality assumption of $\bX$. Specifically, one may instead consider the following assumptions
\begin{assumption}[Spherical Distribution of $\bX$ and sub-Gaussian of $Y$]\label{spherical:sub:gauss}
Let $\bX$ be a spherically distributed $p$ dimensional random variable with $\e[\bX] = 0, \var[\bX] = \Ibb_{p \times p}$, whose moment generating function exists, and takes the form $\EE\{ \exp(\bt\T\bX) \}\equiv \Psi(\bt\T \bt), \bt \in \RR^p$, and in addition $\Psi : \RR \mapsto \RR$ is such that $\Psi(t) \leq \exp(Ct)$ for some $C > 0$ for all $t \in \RR^+$. In addition, we assume that $Y$ is sub-Gaussian, i.e. $\|Y\|_{\psi_2} \leq K_Y$.  
\end{assumption}
Note that \citep[see Lemma 5.5. e.g.]{vershynin2010introduction} assumption \ref{spherical:sub:gauss} implies $\max_{j \in [p]} \|X_j\|_{\psi_2} < \infty$. Let $K:=\max_{j \in [p]} \|X_j\|_{\psi_2}$. In parallel to Proposition \ref{propnormal}, we have
\begin{proposition}\label{covthresh} In addition to Assumption \ref{spherical:sub:gauss}, we assume that s, $c_0 \ne 0$   and $\beta_{0j} \in \{\pm\frac{1}{\sqrt{s}}, 0\}$ for all $j \in [p]$ and some $s \in \mathbb{N}$. Set the tuning parameter $\nu = \omega K_YK$ for some absolute constant $\omega > 0$. Then there exits an absolute constant $\Upsilon \in \RR$ depending on $c_0, C, K$ and $K_Y$ such that if:
$$
n_{p,s} \geq \Upsilon,
$$
Algorithm \ref{Cov:screening} recovers the signed support of $c_0 \bbeta_0$ with asymptotic probability $1$.
\end{proposition}

\subsection[]{LASSO Algorithm with General $\bSigma$} \label{cov:struct:section}
In this section we consider the LASSO algorithm (\ref{LASSOproblem:intro}), 
\begin{align} \label{LASSOproblem}
\widehat \bbeta = \argmin_{\smbbeta \in \RR^p} \frac{1}{2n} \|\bY - \Xbb \bbeta\|_2^2 + \lambda \|\bbeta\|_1,
\end{align}
under the assumption $\bX \sim \cN(0, \bSigma)$ with a generic and unknown covariance matrix $\bSigma$. Under certain sufficient conditions our goal is to show that (\ref{LASSOproblem:intro}) recovers the support of  $\bbeta_0$ with asymptotic probability $1$ in optimal (up to a scalar) effective sample size. In contrast to section \ref{covariancescreening}, we also no longer require each of the signals of $\bbeta$ to be of the same magnitude. 

We first summarize a primal dual witness (PDW) construction which we borrow from \cite{wainwright2009sharp}. The PDW construction lays out steps allowing one to prove sign consistency for $L_1$ constrained quadratic programming (\ref{LASSOproblem}). We will only provide the sufficient conditions to show sign-consistency, and the interested reader can check \cite{wainwright2009sharp} for the necessary conditions. We note that the validity of the PDW construction is generic, in that it does not rely on the distribution of the residual $\bfw = \bY - c_0 \Xbb \bbeta_0$, and hence extends to the current framework. Note that unlike the linear regression case, in our setting $\bfw$ does not necessarily have mean 0 although $\e [\Xbb\T \bfw] = 0$. 

%Recall that for a vector $\vb \in \RR^p$ we denote its support by $S(\vb) = \{i: v^i \neq 0 \}$. Put $S := S(\bbeta)$ for brevity. As we mentioned we are interested, more generally, in signed support recovery. 
Recall that a vector $\zb$ is a subgradient of the $L_1$ norm evaluated at a vector $\vb \in \RR^p$ (i.e. $\zb \in \partial \|\vb\|_1$) if we have $z_j = \sign(v_j), v_j \neq 0$ and $z^j \in [-1,1]$ otherwise. It follows from Karush-Kuhn-Tucker's theorem that a vector $\widehat \bbeta \in \RR^p$ is optimal for the LASSO problem (\ref{LASSOproblem}) iff there exists a subgradient $\widehat \zb \in \partial \|\widehat \bbeta\|_1$ such that:
\begin{align} \label{KKT}
\frac{1}{n} \Xbb\T \Xbb (\widehat \bbeta - c_0 \bbeta_0) - \frac{1}{n} \Xbb\T \bfw+ \lambda \widehat \zb = 0.
\end{align}
Put $S_0 := S(\bbeta_0)$ for brevity. In what follows we will assume that the matrix $\Xbb_{,S_0}\T \Xbb_{,S_0}$ is invertible (which it is with probability $1$), even though this is not required by the PDW. The PDW method constructs a pair $(\check \bbeta, \check \zb) \in \RR^p \times \RR^p$ by following the steps:
\begin{itemize}
\item Solve:
$$
\check \bbeta_{S_0} = \argmin_{\smbbeta_{S_0} \in \RR^s} \frac{1}{2n} \|\bY - \Xbb_{,S_0}\bbeta_{S_0}\|_2^2 + \lambda \|\bbeta_{S_0}\|_1,
$$
where $s = |S_0|$. This solution is unique under the invertibility of $\Xbb_{,S_0}\T \Xbb_{,S_0}$ since in this case the function is strictly convex. Set $\check \bbeta_{S_0^c} = 0$. 
\item Choose a $\check \zb_{S_0} \in \partial \|\check \bbeta_{S_0}\|_1$.
\item For $j \in S_0^c$ set $Z_j := \Xbb_{,{j}}\T [\Xbb_{,S_0} (\Xbb_{,S_0}\T \Xbb_{,S_0})^{-1} \check \zb_{S_0} + \Pb_{\Xbb_{,S_0}^{\perp}}\left(\frac{\bfw}{\lambda n}\right)]$\footnote{$Z_j$ are derived by simply plugging in $\check \bbeta$ and $\check \zb_{S_0}$ and solving  (\ref{KKT}) for $\check \zb_{S_0^c}$.}, where $\Pb_{\Xbb_{,S_0}^{\perp}} = \Ibb - \Xbb_{,S_0} (\Xbb_{,S_0}\T \Xbb_{,S_0})^{-1}\Xbb_{,S_0}\T$ is an orthogonal projection. Checking that $|Z_j| < 1$ for all $j \in S_0^c$ ensures that there is a unique solution $\check \bbeta = (\check \bbeta_{S_0}\T, \check \bbeta\T_{S_0^c})\T$ satisfying $S(\check \bbeta) \subseteq S(c_0\bbeta_0)$. 
\item To check sign consistency we need $\check \zb_{S_0} = \sign(c_0 \bbeta_{0S_0})$. For each $j \in S_0$, define:
$$
\Delta_j:= \eb_j\T (n^{-1} \Xbb_{,S_0}\T \Xbb_{,S_0})^{-1} \left[n^{-1}\Xbb_{,S_0}\T \bfw - \lambda \sign(c_0 \bbeta_{0S_0})\right], \footnote{$\Delta_j$ can be seen to equal $\check \beta_j - c_0 \beta_{0j}$ for $j \in S_0$, when $\check \zb_{S_0} = \sign(c_0 \bbeta_0)$.}
$$
where $\eb_j \in \RR^s$ is a canonical unit vector with 1 at the $j$\textsuperscript{th} position. Checking $\check \zb_{S_0} = \sign(c_0 \bbeta_{0S_0})$ is equivalent to checking:
$$
\sign(c_0 \beta_{0j} + \Delta_j) = \sign(c_0\beta_{0j}), ~ \forall j \in S_0.
$$
\end{itemize}
To this end we require several restrictions on $\bSigma$ and moment conditions on $Y$. We partition the covariance matrix 
$$\bSigma = \left[\begin{array}{cc}\bSigma_{S_0,S_0} & \bSigma_{S_0,S_0^c}\\ \bSigma_{S_0^c,S_0} & \bSigma_{S_0^c,S_0^c}\end{array}\right],$$
where $\bSigma_{S_0,S_0}$ corresponds to the covariance of $\bX_{S_0}$. Furthermore, we let
\begin{align}
\bSigma_{S_0^c|S_0} &:= \bSigma_{S_0^c,S_0^c} - \bSigma_{S_0^c,S_0} \bSigma_{S_0,S_0}^{-1} \bSigma_{S_0,S_0^c},\\
\rho_{\infty}(\bSigma_{S_0,S_0}^{1/2}) & := \|\bSigma_{S_0,S_0}^{-1/2}\|_{\infty, \infty}\|\bSigma_{S_0,S_0}^{1/2}\|_{\infty, \infty},
\end{align}
be the conditional covariance matrix of $\bX_{S_0^c} | \bX_{S_0}$, and the condition number of $\bSigma_{S_0,S_0}^{1/2}$ with respect to $\|\cdot\|_{\infty,\infty}$, respectively. 
We assume that
\begin{assumption}[Irrepresentable Condition] \label{irrepcond}  
$$
\|\bSigma_{S_0^c,S_0}\bSigma_{S_0,S_0}^{-1}\|_{\infty,\infty}  \leq (1 - \kappa), \quad \mbox{for some $\kappa > 0$.}
$$
\end{assumption}
\begin{assumption}[Bounded Spectrum]\label{bounded:spec:assump} For some fixed $0 < \lambda_{\min} \leq \lambda_{\max} < \infty $,
$$
 \lambda_{\min} \leq \bSigma_{S_0,S_0} \leq \lambda_{\max} .
$$
\end{assumption}
\begin{assumption}[Bounded 4\textsuperscript{th} Moment] \label{bounded4} $\e(Y^4) < \infty$, and does not scale with $(n,p,s)$.
\end{assumption}

Recall that the irrepresentable condition is proved to be necessary for successful support recovery \cite[Theorem 4]{wainwright2009sharp} in the linear model, and hence one should not expect that Assumption \ref{irrepcond} can be weakened. Assumption \ref{bounded4} guarantees that $\sigma^2$, $\eta$, $c_0$, $\gamma$, $\xi^2$ and $\theta^2$ 
are well defined and finite. Finally, successful support recovery will depend on the strength of the minimal signal of $\bbeta_0$. Recall that $\bbeta_0^{\min} := \min_{j \in S_0}|\beta_{0j}|$, is the minimal non-zero signal in the vector $\bbeta_0$. We are now ready to provide sufficient conditions for the LASSO signed support recovery, in the setting of SIMs

\begin{theorem} \label{mainresult} %In addition to Assumptions \ref{irrepcond} --- \ref{bounded4}, we assume that $s = O(p^{1- \omega})$ for some $\omega > 0$. Then for the
Let Assumptions \ref{irrepcond} --- \ref{bounded4} hold. Then for the
LASSO estimator given in (\ref{LASSOproblem}) under a SIM (\ref{SIM}), we have the following sufficient conditions:
\begin{itemize}
\item[i.] If 
$$
n_{p,s} \geq \frac{4D_{\max}(\bSigma_{S_0^c| S_0}) \left(\frac{4}{\lambda_{\min}} + \frac{\xi^2 + 1}{\lambda^2s}\right)}{\kappa^2},
$$
then $S(\widehat \bbeta) \subset S(c_0 \bbeta_0)$, with probability at least $1 - O\{p^{-1} + n^{-1} + e^{-\Omega(s)}\}$. 
%\tcomm{why not use $1 - O\{p^{-1} + n^{-1} + e^{- c_0 s}\}$ for some constant $c_0 > 0$}
%\mcomm{we could but it is equivalent; there are too many constants as it is, becomes too messy (we also have the $\Upsilon$ etc)}
%%
\item[ii.] Let further, for some positive constant $\alpha > 0$ we have $n_{p,s} \geq \alpha$. Then there exist some positive constants $\Upsilon_0, \Upsilon_1, \Upsilon_2 > 0$ which may depend\footnote{The dependency is inversely proportional to $|c_0|$ and proportional to $\sigma$; For more details refer to the proof.} on $c_0, \sigma$ such that if:
\begin{align} \label{min:signal:strength}
\bbeta_0^{\min} & \geq  \|\bSigma^{-1/2}_{S_0,S_0}\|^2_{\infty,\infty} \lambda \Upsilon_0 +  \left[\Upsilon_1 \rho_{\infty}(\bSigma^{1/2}_{S_0,S_0})\sqrt{\frac{s}{n \log (p-s)}} \|\bbeta_0\|_{\infty} + \frac{\Upsilon_2 \|\bSigma^{-1/2}_{S_0,S_0}\|_{\infty,\infty}}{\sqrt{s}} \right] n_{p,s}^{-1/2} %+ \rho_{\infty}(\bSigma^{1/2}_{S_0,S_0})\|\bbeta\|_{\infty}\sqrt{\frac{\log p}{n}}\Upsilon_3,
\end{align}
we have $S_{\pm}(\widehat \bbeta) = S_{\pm}(c_0 \bbeta_0)$ with probability at least $1 - O\{e^{-\Omega( s \wedge \log (p-s))} + (\log p)^{-1} + n^{-1}\}$.
\end{itemize}
\end{theorem}
Before we proceed with the proof of our main result, we would like to mention a few remarks on our sufficient conditions, in particular the ones suggested in ii.

\begin{remark}\label{signal:strength:remark} Firstly, the slow probability convergence rate $(\log p)^{-1}$ in part ii. is purely due to the fact we are not willing to assume that $Y$ is coming from a sub-Gaussian distribution. If such an assumption is made the rate reduces to the usual $p^{-1}$.
Secondly, observe that $\lambda_{\max}^{-1} \leq \|\bbeta_0\|_2 \leq \lambda_{\min}^{-1}$. Hence the value of $\bbeta_0^{\min}$ is of ``largest'' order when $\bbeta_0^{\min} \asymp \|\bbeta_0\|_{\infty} \asymp \frac{1}{\sqrt{s}}$. Setting 
$$
\lambda := \lambda_T = \sqrt{(\xi^2 + 1)\frac{4 C_T D_{\max}(\bSigma_{S_0^c|S_0})}{\kappa^2}\frac{\log(p-s)}{n}},
$$ 
for some $C_T > 1$ gives us that the condition from i. is equivalent to:
$$
n_{p,s} \geq \frac{16 D_{\max}(\bSigma_{S_0^c|S_0})}{(1 - C^{-1}_T)\kappa^2 \lambda_{\min}}.
$$
Note that due to positive definiteness: $D_{\max}(\bSigma_{S_0^c|S_0}) \leq D_{\max}(\bSigma_{S_0^c,S_0^c})$, and hence it is reasonable to assume $D_{\max}(\bSigma_{S_0^c|S_0}) = O(1)$. Assume additionally that $\|\bSigma^{-1/2}_{S_0,S_0}\|_{\infty,\infty} = O(1)$, $\rho_{\infty}(\bSigma^{1/2}_{S_0,S_0}) = O(1)$, and let $\bbeta_0^{\min} \asymp \frac{1}{\sqrt{s}}$. Notice that the space of matrices satisfying assumptions \ref{irrepcond} --- \ref{bounded4} and the latter assumptions is non-empty as one can easily show that Toeplitz matrices with entries $\Sigma_{kj} = \rho^{|k - j|}$ with $|\rho| < 1$ satisfy these conditions e.g. Using the same $\lambda = \lambda_T$ we can clearly achieve the sufficient condition in ii. provided that the model complexity adjusted sample size $n_{p,s}$ is large enough. On the other hand, this scaling can no longer be guaranteed if $\bbeta_0^{\min} \asymp \frac{1}{\sqrt{s}}$ fails to hold. In any case, it is clear that when $\bSigma = \eye_{p \times p}$ all conditions required are met, and hence Theorem \ref{mainresult} shows that the LASSO algorithm will work optimally (up to a scalar) in terms of the rescaled sample size. % for the class of models described in Proposition \ref{lower:bound:prop}. 
Below we formulate a result which allows us to claim certain optimality for the LASSO algorithm in the case of a more general $\bSigma$ matrix whenever we have approximately equally sized coefficients.
\end{remark}

\begin{proposition} \label{lower:bound:prop:new} Consider a special example of a SIM with $Y = G\{h(\bX\T \bbeta_0) + \varepsilon\}$ for  $\bX \sim \cN(0, \bSigma)$, and $G, h$ are known strictly increasing continuous functions and in addition $h$ is an $L$-Lipschitz function. Assume that there exists a set $S \subset [p], |S| = s$ such that $\|\bSigma_{S,S}\|_{\infty,\infty} < R$, Assumptions \ref{irrepcond} and \ref{bounded:spec:assump} hold on $S$ and $0 < d \leq \diag(\bSigma_{S^c,S^c}) \leq D < \infty$. In addition, assume that $\varepsilon$ is a continuously distributed random variable with density $p_{\varepsilon}$ satisfying:%as specified in (\ref{eps:dens}).
\begin{align}\label{eps:dens}
p_\varepsilon(x) \varpropto \exp(-P(x^2)),
\end{align}
where $P$ is any non-zero polynomial with non-negative coefficients such that $P(0) = 0$. We restrict the parameter space to $\Big\{\bbeta \in \RR^p: \bbeta\T\bSigma\bbeta = 1,  \|\bbeta\|_0 = s, \frac{\bbeta^{\min}}{\|\bbeta\|_{\infty}} \geq c_{\bSigma}\Big\},$ where $c_{\bSigma} > 0$ is a sufficiently small constant depending solely on $\bSigma$ (see (\ref{c_sigma:def}) for details). If 
$$n_{p,s} < C,$$ 
for some constant $C > 0$ (depending on $P, G, h, \bSigma$) and $s$ is sufficiently large, any algorithm for support recovery makes errors with probability at least $\frac{1}{2}$ asymptotically.
\end{proposition}

Proposition \ref{lower:bound:prop:new} justifies that the LASSO is sample size optimal even when more generic covariance matrices than identity are considered, provided that we can show the class of SIMs defined satisfy the assumptions of this section. We do so in the following Remark.

\begin{remark} \label{correlation:remark}  We will now argue that if we have a model as described in Proposition \ref{lower:bound:prop:new}, the LASSO will recover the support provided that the the covariance matrix $\bSigma$ satisfies Assumptions \ref{irrepcond} and \ref{bounded:spec:assump}, $\EE (Y^4) < \infty$ and the minimal signal strength is sufficiently strong as required by Remark \ref{signal:strength:remark}. Notice that by Chebyshev's association inequality \citep{boucheron2013concentration}, we have:
$$
\EE (Y\bX\T\bbeta_0) > \EE(Y)\EE(\bX\T\bbeta_0)  = 0,
$$
where the inequality is strict since $G$ and $h$ are strictly increasing and $\bX\T\bbeta_0 \sim \cN(0,1)$. In fact, using exactly the same argument, one can show more generally that if $r(z) = \EE [Y | \bX\T\bbeta_0= z]$ is a strictly monotone function it follows that $\EE(Y\bX\T\bbeta_0) \neq 0$. 
We close this remark by pointing out that the logistic regression model $P(Y=1 \mid\bX) = g(\bX\T\bbeta_0)$ with $g(x)=e^x/(1+e^x)$ satisfies the condition since $r(z) = g(z)$ is strictly monotone, and hence using the LASSO algorithm one can recover the support correctly. This is an example that even discrete valued $Y$ outcomes can be solved by the least squares LASSO algorithm.  %\tcomm{in simulation you refer to this remark to suggest the benefit of transformation, but it's not clear to me how the remark suggests so}
\end{remark}

\subsubsection{Outcome Transformations}\label{outcome:transf:sec}

In this subsection we provide brief comments on possible strategies to transform the data in view of the results of  Theorem \ref{mainresult}. A crucial condition in order for the signed support recovery of $\bbeta_0$ to hold is $c_0 \neq 0$, which should not be expected to hold in general, but nevertheless naturally occurs in many cases. If this condition does not hold, one can potentially transform the outcome $\tilde Y = g(Y)$ by a function $g$ in order to achieve $\EE(\tilde{Y}\bX\T\bbeta_0) \neq 0$ even if $\EE(Y\bX\T\bbeta_0) = 0$. If we use $\tilde Y_i = g(Y_i)$ instead of $Y_i$ then clearly LASSO succeeds under the assumptions from Theorem \ref{mainresult}, only with assumptions on $Y_i$ being replaced by $\tilde Y_i$. The following proposition characterizes when one should expect a correlation inducing transformation $g$ to exist.
\begin{proposition}\label{rao:blackwell:type:of:res} There exists a measurable function $g : \RR \mapsto \RR$ such that $\EE \{g(Y)\bX\T\bbeta_0\} \neq 0$ if and only if
$\Var\{\EE(\bX\T\bbeta_0 | Y)\} > 0.$
\end{proposition}
Another potential advantage of performing a transformation is to ensure that $\tilde{Y}=g(Y)$ is sub-Gaussian. For example, if we let $g(y) = F(y) = P(Y \le y)$, then the sub-Gaussianity of $\tilde{Y}$ is guaranteed, which would improve the rate of support recovery. %\tcomm{please add in more comments if necessary about the motivation for using $F(\cdot)$.}
For many choices of $g$ such as $F$, the transformation may be defined at the population level and is unknown a priori. Thus it would be desirable to employ data dependent estimate $\hat g$ of $g$. In other words we consider fitting the following LASSO to recover the support:
\begin{align} \label{LASSOproblem:transformed}
\widehat \bbeta = \argmin_{\smbbeta \in \RR^p} \frac{1}{2n} \|\hat g(\bY) - \Xbb \bbeta\|_2^2 + \lambda \|\bbeta\|_1,
\end{align}
where $\hat g(\bY)$ should be understood as element-wise application of $\hat g$. The following Corollary extends Theorem \ref{mainresult} to allow for data dependent transformations.

\begin{corollary}\label{cor:mainres:transform} Let the assumptions of Theorem \ref{mainresult} hold for $\tilde Y_i = g(Y_i)$ in place of $Y_i$ and assume additionally that:
\begin{align}\label{g:concentration}
\|\hat g(\bY) - g(\bY)\|_{2} \leq O(\sqrt{\log p}),
\end{align}
with probability at least $1 - O(p^{-1})$, then if (\ref{min:signal:strength}) holds we have $S_{\pm}(\widehat \bbeta) = S_{\pm}(c_0 \bbeta_0)$ with probability at least 
$1 - O\{e^{-\Omega( s \wedge \log (p-s))} + (\log p)^{-1} + n^{-1}\}$.
\end{corollary}

\begin{remark} Akin to (\ref{LASSOproblem:intro}), we do not require an intercept in the model after doing a transformation in (\ref{LASSOproblem:transformed}). This is possible since $\bX$ is assumed to have mean $0$. In practice if this were not the case, one would have to either center $\bX$ or include an intercept which is not penalized. 
\end{remark}

\section{Proof of Theorem \ref{mainresult}} \label{proof:main:result}

Our proof follows similar steps as Theorem 3 of \cite{wainwright2009sharp} although many critical modifications are needed, since the error term $\wb = \Yb - c_0 \Xbb \bbeta_0$ is no longer independent of $\Xbb$ and is not mean $0$. %\tcomm{make it sound less trivial}

\subsection{Verifying Strict Dual Feasibility} For $j \in S_0^c$ decompose
\begin{align}\label{Ej:decomp}
\Xbb_{,{j}}\T = \bSigma_{\{j\},S_0} \bSigma_{S_0,S_0}^{-1} \Xbb_{,S_0}\T + \bE_j\T,
\end{align} 
where the entries of the prediction error vector $\bE_j = (E_{1j}, \ldots, E_{nj})\T \in \RR^n$ are i.i.d. with $E_{ij} \sim \cN(0, [\bSigma_{S_0^c|S_0}]_{jj}), i \in [n]$. In addition, observe that by this construction we have that $\bE_j$ is independent of $\Xbb_{,S_0}$ which can be verified upon multiplication by $\Xbb_{,S_0}$ in (\ref{Ej:decomp}) and taking expectation. Following the definition of $Z_j$ gives us that $Z_j = A_j + B_j$, where:
\begin{align}
A_j &:= \bE_j\T \left[\Xbb_{,S_0}(\Xbb_{,S_0}\T \Xbb_{,S_0})^{-1} \check \zb_{S_0} + \Pb_{\Xbb_{,S_0}^\perp} \left(\frac{\bfw}{\lambda n}\right)\right],\\
B_j &:= \bSigma_{\{j\},S_0}(\bSigma_{S_0,S_0})^{-1}\check \zb_{S_0}.
\end{align}
Under the irrepresentable condition, we have that $\max_{j \in S^c} |B_j| \leq (1 - \kappa)$. Conditional on $\Xbb_{,S_0}$ and $\bvarepsilon$ (which determine $\bfw = \bY - c_0 \Xbb \bbeta_0$) we have that the gradient $\check \zb_{S_0}$ is independent of the vector $\bE_j$ because the gradient is deterministic after conditioning on these quantities. We have that $\var(E_{ij}) \leq D_{\max}(\bSigma_{S_0^c|S_0})$, and thus conditionally on $\Xbb_{,S_0}$ and $\boldsymbol{\varepsilon}$ we get:
\begin{align*}
\var(A_j) & \leq D_{\max}(\bSigma_{S_0^c|S_0})\left\| \Xbb_{,S_0}(\Xbb_{,S_0}\T \Xbb_{,S_0})^{-1} \check \zb_{S_0} + \Pb_{\Xbb_{,S_0}^\perp} \left(\frac{\bfw}{\lambda n}\right)\right\|_2^2\\
&  = D_{\max}(\bSigma_{S_0^c|S_0})\left[ \check \zb_{S_0}\T (\Xbb_{,S_0}\T \Xbb_{,S_0})^{-1} \check \zb_{S_0} + \left\|\Pb_{\Xbb_{,S_0}^\perp} \left(\frac{\bfw}{\lambda n}\right)\right\|_2^2\right].
\end{align*}
Next we need a lemma, which is a slight modification of Lemma 4 in \cite{wainwright2009sharp}. The reason for this modification is that in our case $\bfw$ is no longer $\sim \cN(0, \sigma^2\Ibb)$.
\begin{lemma} \label{newlemma4} Assume that $\frac{s}{n} \leq \frac{1}{16}$. Then we have:
$$
\max_{j \in S_0^c}\var(A_j) \leq \underbrace{D_{\max}(\bSigma_{S_0^c|S_0})\left(\frac{4 s}{\lambda_{\min} n} + \frac{\xi^2 + 1}{\lambda^2 n}\right)}_{M},
$$
with probability at least $1 - 2 e^{-s/2} - n^{-1}\theta^2 $.
\end{lemma}
Now since conditionally on $\Xbb_{,S_0}$ and $\boldsymbol{\varepsilon}$ we have $A_j \sim \cN(0, \var(A_j))$, using a standard normal tail bound and the union bound we conclude:
$$
\p(\max_{j \in S_0^c} |A_j| \geq \kappa) \leq 2(p-s)e^{-\frac{\kappa^2}{2M}} +2 e^{-\frac{s}{2}} + n^{-1} \theta^2.
$$
We need to select $M$ so that the exponential term is decaying in the above display. A sufficient condition for this is $\kappa^2/(2M) \geq 2 \log(p-s)$. The last is equivalent to:
$$
n_{p,s} \geq \frac{4D_{\max}(\bSigma_{S_0^c| S_0}) \left(\frac{4}{\lambda_{\min}} + \frac{\xi^2 + 1}{\lambda^2 s}\right)}{\kappa^2 }.
$$

\subsection{Verifying Sign Consistency} The first part of the proof shows that the LASSO has a unique solution $\hat \bbeta$ which satisfies $S(\hat \bbeta) \subseteq S(c_0\bbeta_0)$ with high probability. Now we need to verify the sign-consistency, in order to show that the supports coincide. We have the following:
$$
\max_{j \in S_0} |\Delta_j| \leq \underbrace{\lambda \left\| (n^{-1} \Xbb_{,S_0}\T \Xbb_{,S_0})^{-1} \sign(c_0 \bbeta_{0S_0})\right\|_{\infty}}_{\bI_1} + \underbrace{\left\|(\Xbb_{,S_0}\T \Xbb_{,S_0})^{-1} \Xbb_{,S_0}\T \bfw \right\|_{\infty}}_{\bI_2}.
$$
To deal with the first term we need the following:
\begin{lemma}\label{lemmawainwrig} There exist positive constants $K_1, C_2 > 0$, such that the following holds:
$$
\p(\bI_1 \geq \lambda K_1 \|\bSigma^{-1/2}_{S_0,S_0}\|^2_{\infty,\infty}) \leq 4 \exp(- C_2(s \wedge \log(p-s))),
$$
\end{lemma}
The proof of this lemma is part of the proof of Theorem 3 in \cite{wainwright2009sharp} and we omit the details. Next we turn to bounding the term $\bI_2$. Here our proof departs substantially from the proof in \cite{wainwright2009sharp}, as $\bI_2$ no longer has a simple structure required in the original argument. In our case $\bfw$ depends on $\Xbb_{,S_0}$, and it is not mean $0$. We will make usage of the following result, whose proof is provided in the appendix

%\begin{lemma} \label{superimportantlemma} Let $\|\bbeta_0\|_2 = 1$. We have $n$ i.i.d. observations $Y = f(\bX\T\bbeta_0, \varepsilon)$ from a SIM, where $\bX \sim \cN(0,\Ibb_{s\times s})$, with $s < n$. Then there exist some absolute constants $\Upsilon_1, \Upsilon_2, \Upsilon_3 > 0$ (depending on $\sigma$ and $|c_0|$), such that:
%\begin{align*}
%\|[\Xbb\T \Xbb]^{-1} \Xbb\T \bY - c_0 \bbeta_0\|_{\infty} & \leq \Big(\Upsilon_1 \|\bbeta_0\|_{\infty}  + \frac{\Upsilon_2}{\sqrt{s}}\Big)\sqrt{\frac{s \log(p-s)}{n}} + \Upsilon_3 \|\bbeta_0\|_{\infty}\sqrt{\frac{\log(p - s)}{n}}.
%\end{align*}
%with probability at least $1 - \Omega(\exp(-\Omega((s \wedge \log(p - s))))) - \Omega((\log p)^{-1})$. Denote for brevity the RHS of the inequality as $\delta(\|\bbeta_0\|_{\infty}, n, s, p)$.
%\end{lemma}
\begin{lemma} \label{superimportantlemma} %\mcomm{yes; this lemma focuses only on the matrix $\bX \sim \cN(0, \eye_{s \times s})$ on the correct support. the identity case is handled below. I agree it is confusing; if  we change the notation all the proofs should change as well}
%\tcomm{too confusing, can you use $\bX_{S_0}$ and $\bbeta_{0S_0}$ and $\Xbb_{,S_0}$ instead? plz change proof as well} \\
Let $\|\bbeta_0\|_2 = 1$. We have $n$ i.i.d. observations $Y = f(\bX_{S_0}\T\bbeta_{0S_0}, \varepsilon)$ from a SIM, where $\bX_{S_0} \sim \cN(0,\Ibb_{s\times s})$, with $s < n$ and $n_{p,s} \geq \alpha > 0$ for some positive constant $\alpha$. Then there exist some positive constants $\Upsilon_1, \Upsilon_2> 0$ (depending on $\sigma$ and $|c_0|$), such that:
\begin{align*}
\|[\Xbb_{,S_0}\T \Xbb_{,S_0}]^{-1} \Xbb_{,S_0}\T \bY - c_0 \bbeta_{0S_0}\|_{\infty} & \leq \Big(\Upsilon_1 \frac{s}{n} \|\bbeta_{0S_0}\|_{\infty} + \Upsilon_2\sqrt{\frac{\log(p-s)}{n}}\Big). %+ \Upsilon_3 \|\bbeta_0\|_{\infty}\sqrt{\frac{\log(p - s)}{n}}.
\end{align*}
with probability at least $1 - O\{e^{-s/2} + (\log p)^{-1} + n^{-1} + p^{-1}\}$. Denote for brevity the RHS of the inequality as $\delta(\|\bbeta_{0S_0}\|_{\infty}, n, s, p)$.
\end{lemma}

While Lemma \ref{superimportantlemma} is stated in terms of standard multivariate normal distribution $\cN(0, \Ibb_{s \times s})$, we can easily adapt it to more general situations where we observe non-standard normal random variables $\cN(0, \bSigma_{S_0,S_0})$. Recall that the rows of $\Xbb_{,S_0}$ are distributed as $\cN(0, \bSigma_{S_0,S_0})$, $Y_i = f(\bX_i\T\bbeta_0, \varepsilon)$, and $\bbeta_{0S_0}\T \bSigma_{S_0,S_0} \bbeta_{0S_0} = 1$. Denote with $\bfZ = \Xbb_{,S_0} \bSigma^{-1/2}_{S_0,S_0}$. Then we have the following inequality, with high probability:
\begin{align*}
\bI_2 = \|[\Xbb\T_{,S_0} \Xbb_{,S_0}]^{-1} \Xbb_{,S_0}\T \bY - c_0 \bbeta_{0S_0}\|_{\infty} & = \|\bSigma^{-1/2}_{S_0,S_0}[\bfZ\T \bfZ]^{-1} \bfZ\T \bY - c_0 \bbeta_{0S_0}\|_{\infty}\\
&  \leq \|\bSigma^{-1/2}_{S_0,S_0}\|_{\infty, \infty} \|[\bfZ\T \bfZ]^{-1}  \bfZ\T \bY - c_0 \bSigma^{1/2}_{S_0,S_0} \bbeta_{0S_0}\|_{\infty}\\
& \leq  \|\bSigma^{-1/2}_{S_0,S_0}\|_{\infty, \infty} \delta(\|\bSigma^{1/2}_{S_0,S_0}\bbeta_{0S_0}\|_{\infty}, s, n, p).
\end{align*}
The last two inequalities imply that:
$$
\max_{j \in S} |\Delta_j| \leq  \lambda K_1 \|\bSigma^{-1/2}_{S_0,S_0}\|^2_{\infty,\infty} + \|\bSigma^{-1/2}_{S_0,S_0}\|_{\infty, \infty} \delta(\|\bSigma^{1/2}_{S_0,S_0}\bbeta_{0S_0}\|_{\infty}, s, n, p).
$$
Hence as long as for $\bbeta_0^{\min} = \min\{|\beta_{0j}|: j \in S_0\}$ we have:
$$
|c_0|\bbeta_0^{\min} \geq  \lambda K_1 \|\bSigma^{-1/2}_{S_0,S_0}\|^2_{\infty,\infty} + \|\bSigma^{-1/2}_{S_0,S_0}\|_{\infty, \infty} \delta( \|\bSigma^{1/2}_{S_0,S_0}\|_{\infty,\infty} \|\bbeta_0\|_{\infty}, s, n, p),
$$
the LASSO will recover the support with high-probability. This concludes the proof.

\section{Numerical Studies} \label{simulations:sec}To support our theoretical claims, and in particular Theorem \ref{mainresult} we provide brief numeric analysis in this section. We consider the following models:
\begin{align}
Y & = \bX\T\bbeta_0 + \sin(\bX\T\bbeta_0) + \cN(0,1), \label{sinmodel}\\
Y & = 2 \operatorname{atan}(\bX\T\bbeta_0) + \cN(0,1), \label{X3model} \\
Y & = (\bX\T\bbeta_0)^3 + \cN(0,1), \label{Xe3model}\\
Y & = \operatorname{sinh}(\bX\T\bbeta_0)+ \cN(0,1). \label{regmodel}
\end{align}
We use a Toeplitz covariance matrix for the simulations $\bSigma = \eye$ and $\Sigma_{kj} = 2^{-|k - j|}$. The vector $\bbeta_0$ is selected so that $\bbeta_0\T\bSigma \bbeta_0 = 1$, and its entries have equal magnitude, with the first one having a negative sign, and the remaining being positive. We check whether the solution path of the LASSO contains an $s$-sparse vector $\bbeta_0$ whose support coincides with the support of $\bbeta_0$. This verifies the validity of one implication of our theory, as it shows that the solution path indeed contains the true signed support of $\bbeta_0$. 

\begin{figure}[H]
\caption{LASSO, $s = \sqrt{p}$, $\bSigma = \eye_{p \times p}$} \label{figuressqrtplasso:sigmaid}
\begin{subfigure}{.5\textwidth}
  \centering
  \includegraphics[width=.7\linewidth]{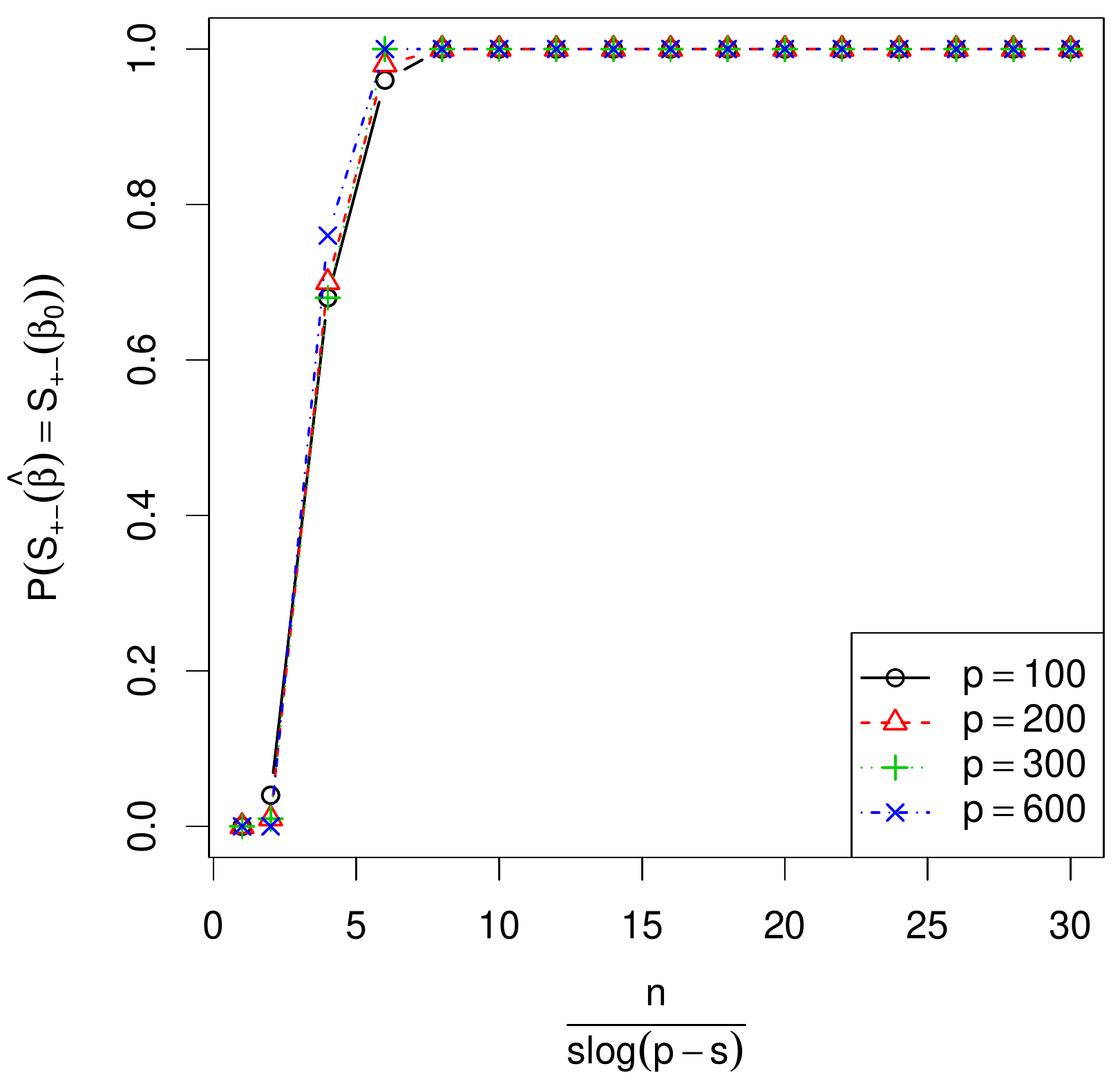}
  \caption{Model (\ref{sinmodel})}
  \label{fig:sfig1}
\end{subfigure}%
\begin{subfigure}{.5\textwidth}
  \centering
  \includegraphics[width=.7\linewidth]{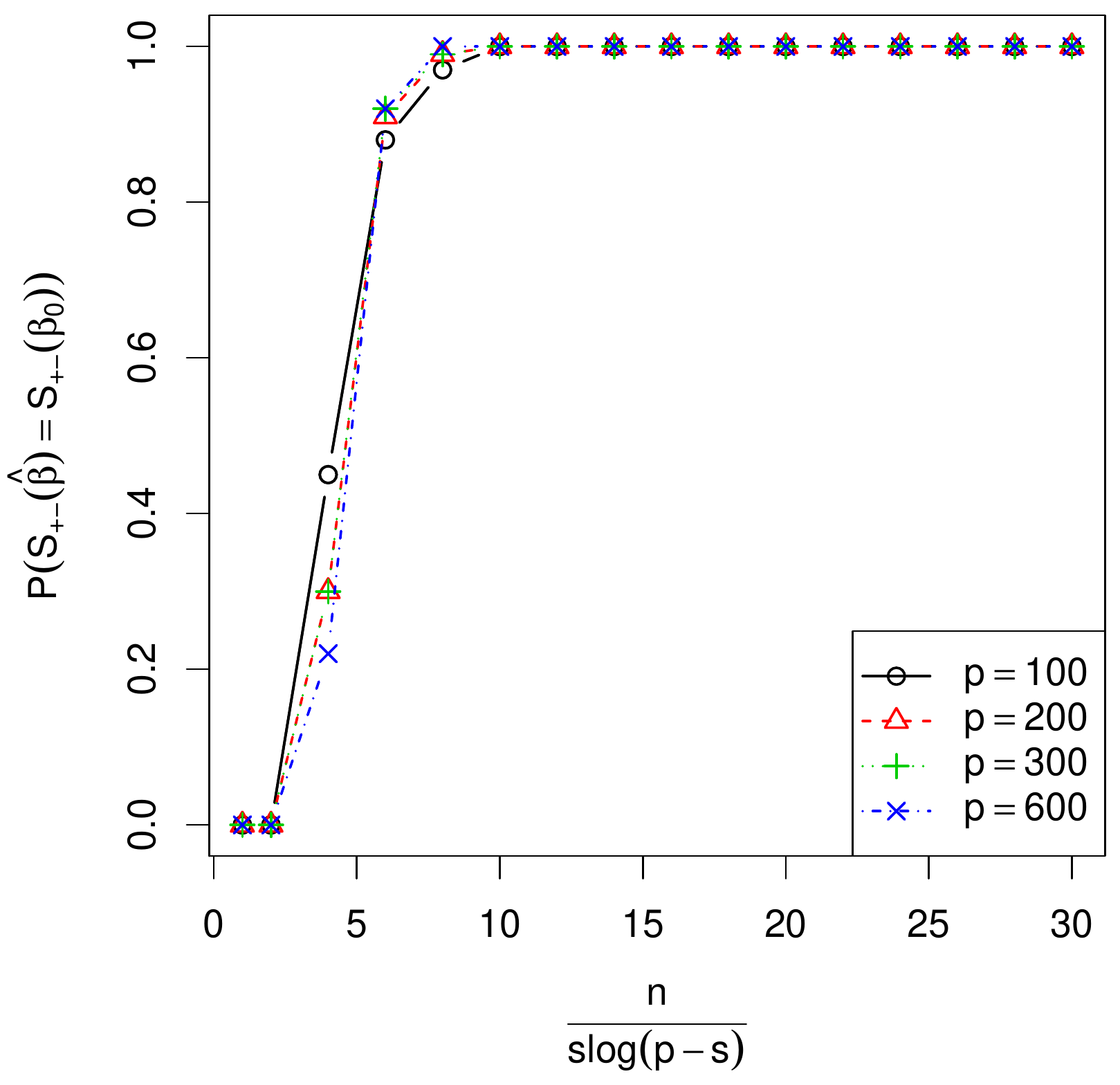}
  \caption{Model (\ref{X3model})}
  \label{fig:sfig2}
\end{subfigure}%

\begin{subfigure}{.5\textwidth}
  \centering
  \includegraphics[width=.7\linewidth]{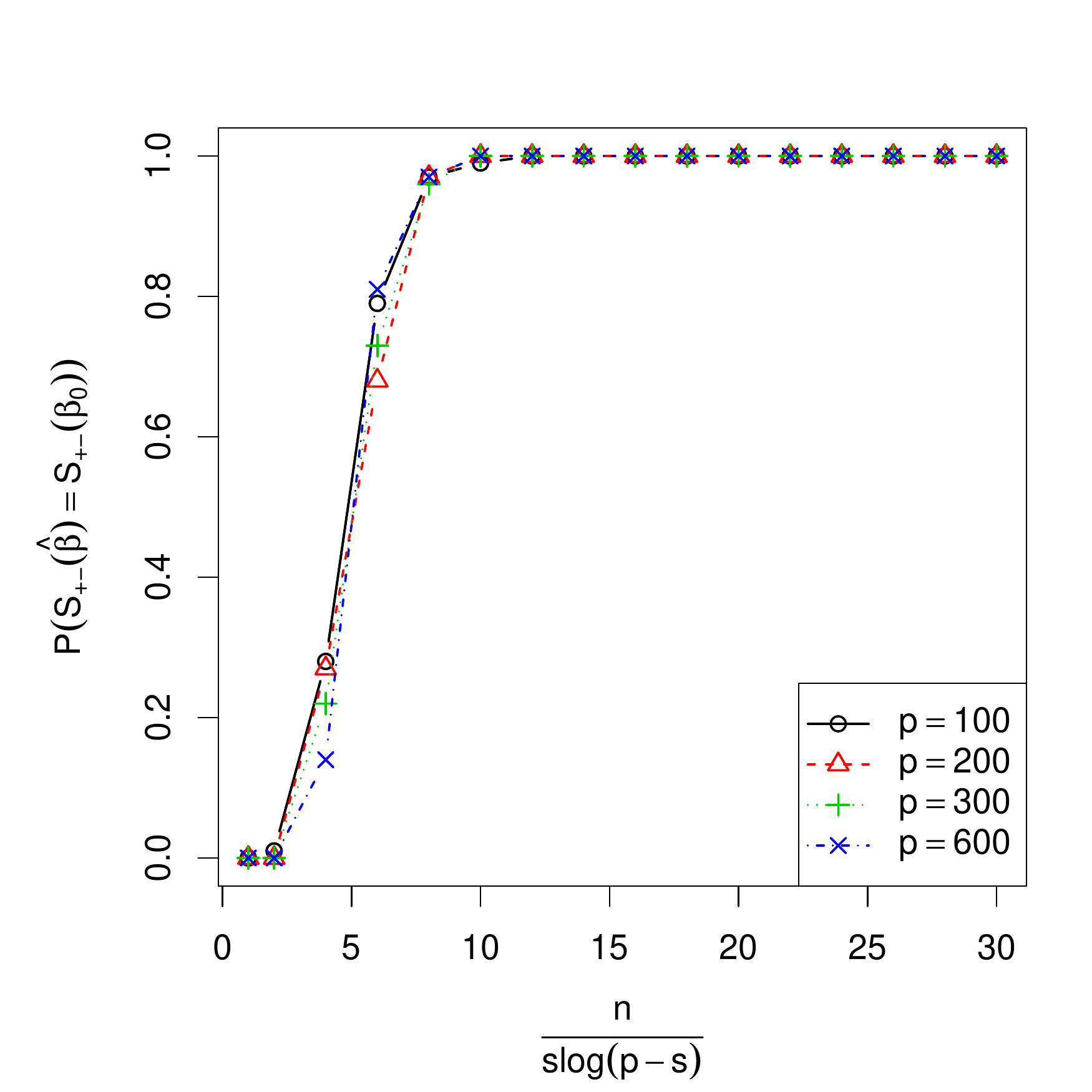}
  \caption{Model (\ref{Xe3model})}
  \label{fig:sfig3}
\end{subfigure}%
\begin{subfigure}{.5\textwidth}
  \centering
  \includegraphics[width=.7\linewidth]{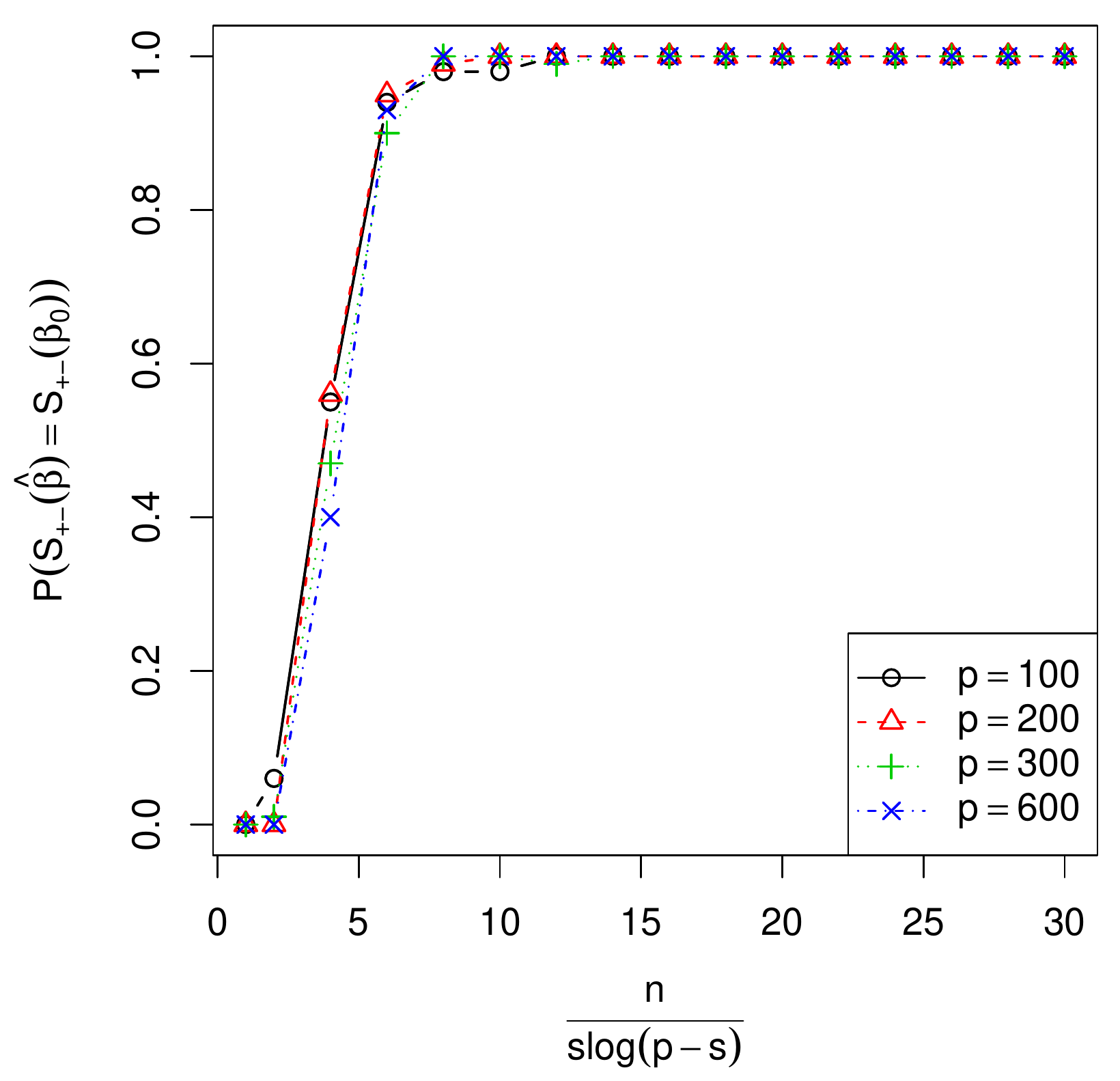}
  \caption{Model (\ref{regmodel})}
  \label{fig:sfig4}
\end{subfigure}%
\end{figure}

In Figure \ref{figuressqrtplasso:sigmaid}, we show results of signed support recovery for different values of $p$, in the regime $s = \sqrt{p}$ in the case of an identity covariance matrix $\bSigma = \eye_{p \times p}$. Similarly, Figure \ref{figuressqrtplasso} shows results for the case of Toeplitz covaraince matrix $\Sigma_{kj} = 2^{-|k - j|}$. As expected, the support recovery is harder in the presence of correlation between the variables. These results illustrate different phase transitions occurring for the four different models. We observe empirically that values of $n_{p,s}$ achieving reasonable success probability can be large in some cases. It could be the case that using a transformed version of $Y$ %(as suggested by Remark \ref{correlation:remark}) 
%\tcomm{not sure how that remark indicates this, or why u believe that transformation would help here. perhaps suggest a $g$ for a model?} 
%\mcomm{This is referring to the wrong remark; Here I meant that it is possible the convergence in prob to be faster if we transform $Y$ to the $[0,1]$ say; It is not obvious that it will help we might skip this comment} \tcomm{if referring to wrong remark, you have to correct this and don't label that remark as transform!!}
might lead to better results for the model complexity adjusted sample size, as suggested in Section \ref{outcome:transf:sec}. Figure \ref{figuressqrtplasso} supports the result of Theorem \ref{mainresult} as all curves merge when the effective sample size $n_{p,s}$ becomes sufficiently large. In addition, these results suggest that the performance of the support recovery is largely determined by $(n,p,s)$ through the magnitude of $n_{p,s}$. %\tcomm{agree?}

%\begin{figure}[H] \caption{model1}
%\centering
%\includegraphics[width=.7\linewidth]{/Users/mateyneykov/Dropbox/LASSO_paper_Jun_Tianxi/some_simulation_res/correct_simu_LASSO_vs_LL_vs_SIR/combined_plots_cov_0.5/model1.pdf}
%\end{figure}
\begin{figure}[H]
\caption{LASSO, $s = \sqrt{p}$, $\Sigma_{kj} = 2^{-|k-j|}$} \label{figuressqrtplasso}
\begin{subfigure}{.5\textwidth}
  \centering
  \includegraphics[width=.7\linewidth]{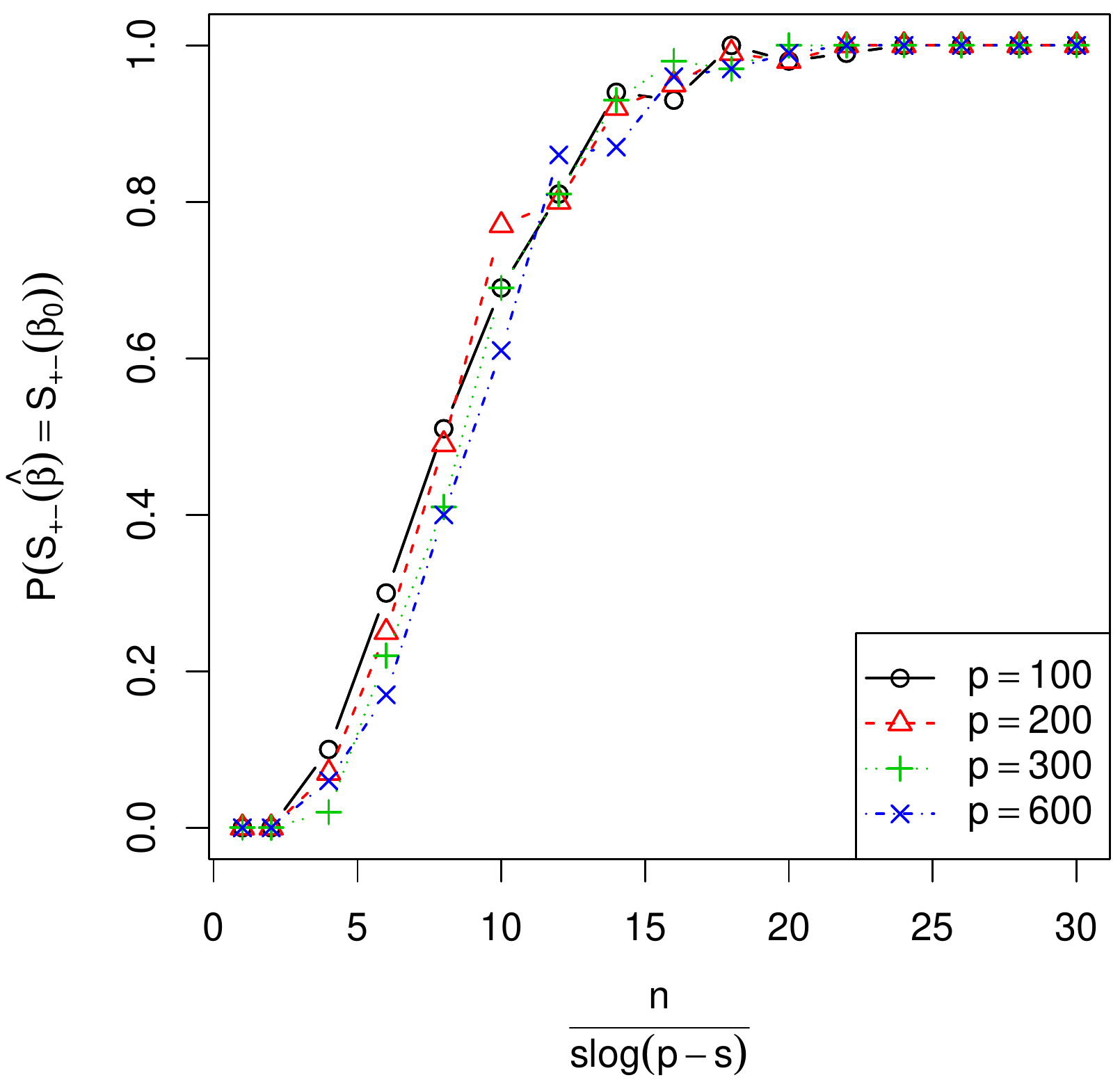}
  \caption{Model (\ref{sinmodel})}
  \label{fig:sfig1}
\end{subfigure}%
\begin{subfigure}{.5\textwidth}
  \centering
  \includegraphics[width=.7\linewidth]{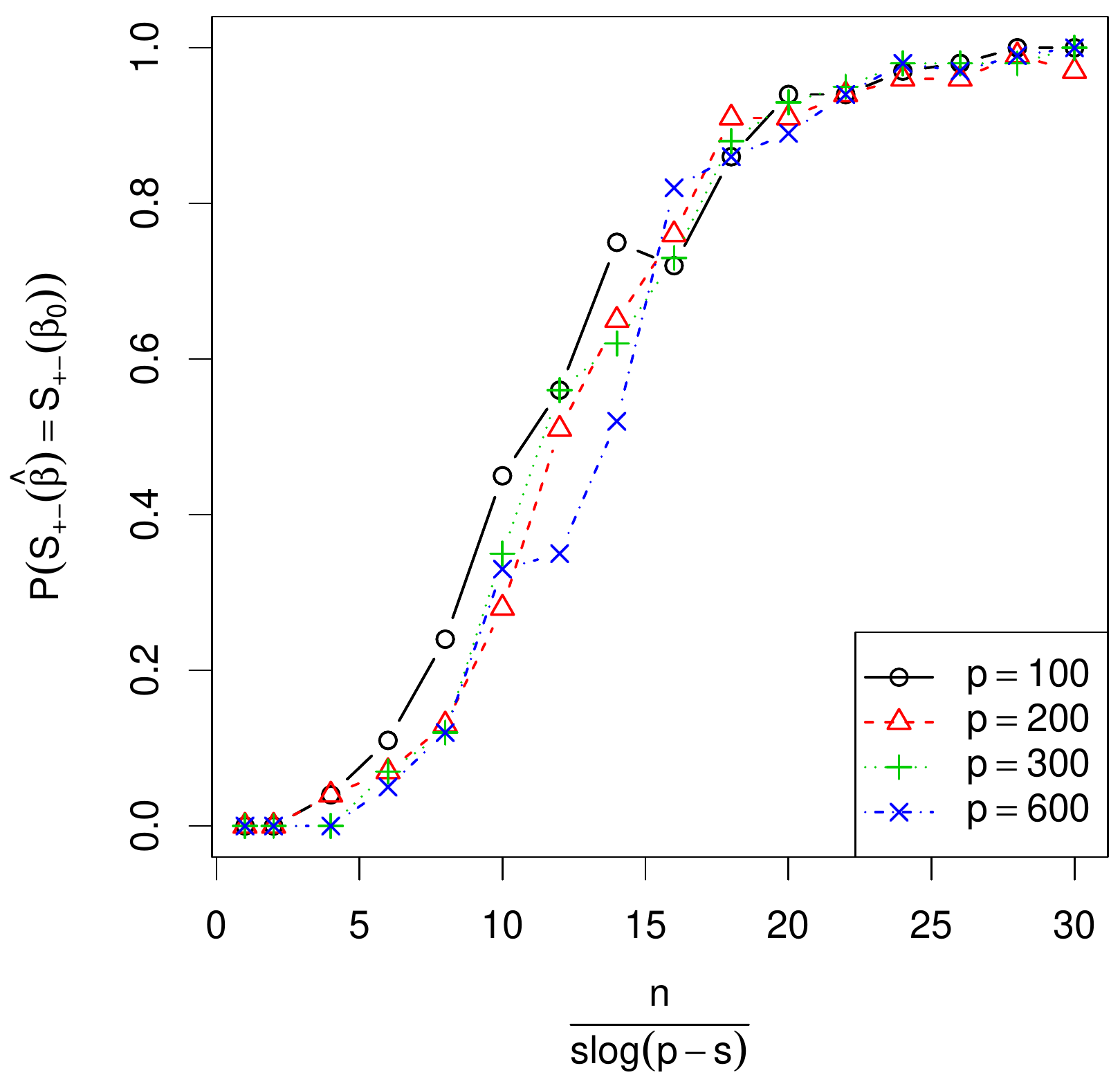}
  \caption{Model (\ref{X3model})}
  \label{fig:sfig2}
\end{subfigure}%

\begin{subfigure}{.5\textwidth}
  \centering
  \includegraphics[width=.7\linewidth]{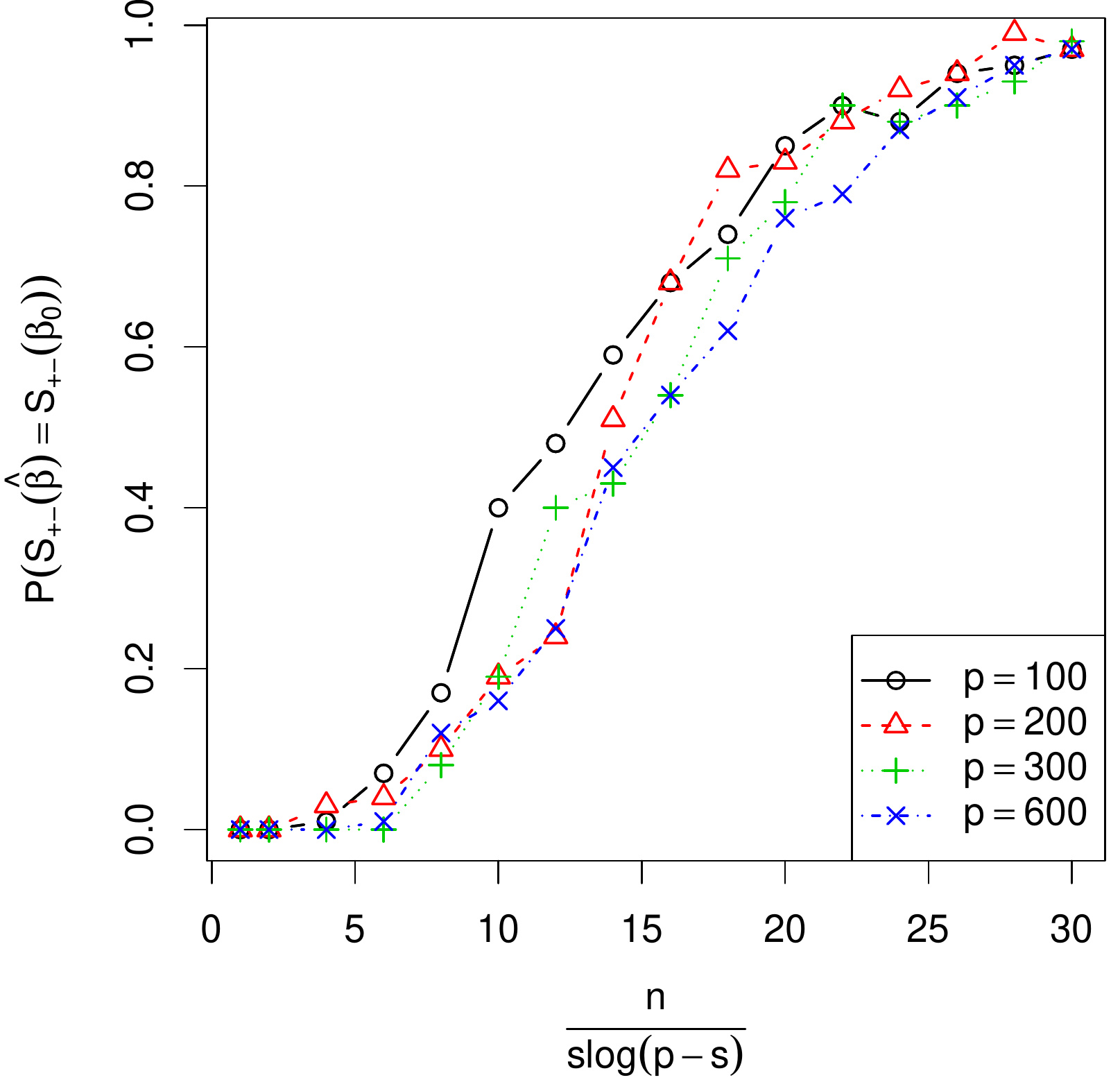}
  \caption{Model (\ref{Xe3model})}
  \label{fig:sfig3}
\end{subfigure}%
\begin{subfigure}{.5\textwidth}
  \centering
  \includegraphics[width=.7\linewidth]{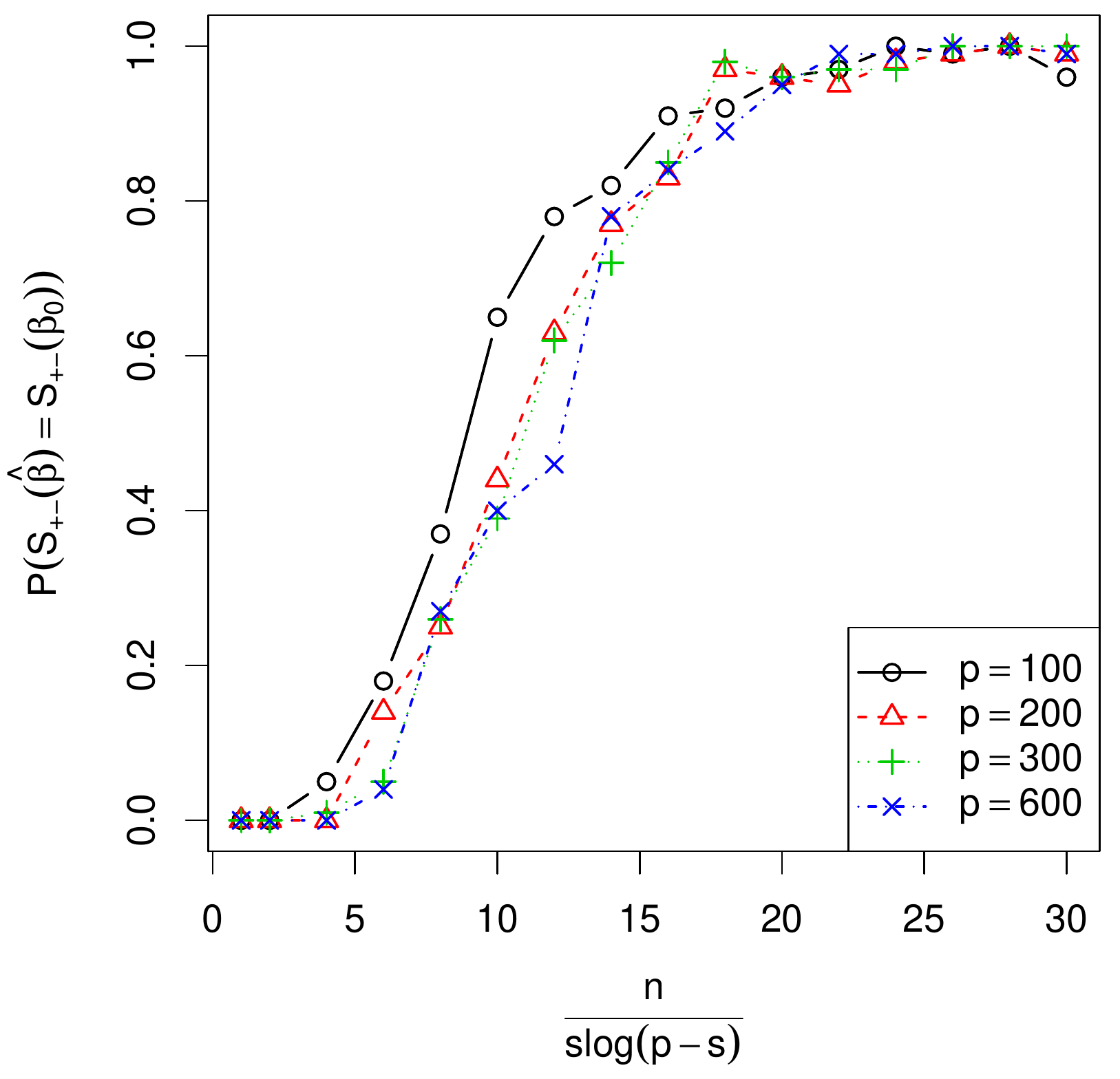}
  \caption{Model (\ref{regmodel})}
  \label{fig:sfig4}
\end{subfigure}%
\end{figure}

In addition to the verification of our theory, we also compare the vanilla least squares LASSO to  a \textit{version} of the sparse sliced inverse regression (SSIR) algorithm suggested by \cite{li2006sparse}. The SSIR algorithm is also based on a LASSO estimation. In its original form however, this algorithm is not applicable for high-dimensional settings such as ours, since it needs an estimate of the matrix $\bSigma^{-1/2}$. To make use of the SSIR under the high-dimensional setting, we estimate $\bSigma^{-1/2}$ by the CLIME procedure \citep{cai2011constrained} to $\hat \bSigma^{1/2}$ under sparsity assumptions. Due to space considerations we only show comparisons for models (\ref{sinmodel}) and (\ref{regmodel}), and the plots are attached in Appendix \ref{more:plots}. In the majority of cases LASSO outperforms the SSIR algorithm substantially for small values of $n_{p,s}$, although it seems that both approaches reach perfect support recovery at similar rescaled sample sizes. We would like to emphasize that the SSIR algorithm requires solving an extra optimization problem, and that furthermore, there are no theoretical results ensuring that SSIR in general recovers the support. On an important note, the SSIR algorithm is designed to estimate the central space of the more general class of multi-index models, which we do not discuss in the present paper. For a brief discussion on how our work can be related to multi-index model please refer to section \ref{discussion:sec}.

%\tcomm{don't think there is any theoretical result for SSIR?} 
%We believe however, that it is possible to derive such guarantees similarly to the way we showed the corresponding properties for the vanilla LASSO.

\section{Discussion}\label{discussion:sec}
%In this paper we analyzed two algorithms for support recovery of high-dimensional SIM --- covariance screening and LASSO, and showed that both can recover the support optimally when the design matrix $\bX \sim \cN(0, \eye_{p \times p})$ and the minimal signal strength is large enough. In addition, we demonstrate that LASSO can work well when $\bX \sim \cN(0, \bSigma)$ provided standard assumptions on $\bSigma$ which are necessary even in the linear regression case. Our results extend known results on the support recovery performance of LASSO under linear models to a much broader class of SIM.
In this paper, we demonstrate that under a high dimensional SIM, $L_1$-regularized least squares, including a simplified covariance screening procedure under orthogonal design, is robust in terms of model selection consistency, in that it correctly recovers the support of the true regression parameter $\bbeta_0$ provided that $c_0 = \EE(Y \bX\T \bbeta_0) \neq 0$, the minimal signal strength is sufficiently large and $\bX \sim \cN(0, \bSigma)$ under standard assumptions on $\bSigma$ which are necessary even in the linear regression case. Thus, our results extend known results on the support recovery performance of LASSO under linear models to a much broader class of SIMs.
We furthermore demonstrate that the support recovery is achieved in a sample size optimal $n_{p,s}$ manner within a certain class of SIMs. 

As we indicated in section \ref{outcome:transf:sec}, the assumption $c_0 \neq 0$ does not always hold, and in addition it cannot be easily verified. A potential remedy for this approach will be to transform the $Y$ variable. From theoretical point of view it is of interest to develop procedures which can adaptively estimate a ``good'' outcome transformation. Additionally, a downside of the $L_1$-regularization is the fact that the irrepresentable condition on the covariance matrix is unavoidable. This could potentially be remedied by using more general and non-convex penalties such as the SCAD penalty \citep{fan2001variable}. We focused on the setting with $\bX \sim \cN(0, \bSigma)$, however we suspect that the support recovery holds in more general cases where $\bX$ comes from an elliptical distribution. It is less clear, however, whether sample size optimality continues to hold in such situations, as we crucially rely on the normality of $\bX$, in particular when using Lemma \ref{cor535versh} which follows by Gordon's comparison theorem, and through numerous projection-independence properties which are characteristic of the Gaussian distribution. %... \tcomm{can you comment on the potential generalizability to non-normal $\bX$?} 

The proposed method focuses on SIMs as the true underlying model. Extensions to incorporate general multi-index models are not straightforward. For the special case of multi-index models of the form
$$
Y = \sum_{j = 1}^k f_j(\bX\T \bbeta_j, \varepsilon_j),
$$
our method should also be able to recover the support, assuming that $k$ is fixed and the vectors $\bbeta_j$ have disjoint supports, and $\EE (Y \bX\T \bbeta_j) \neq 0$ for all $j \in [k]$. When applied to such a model, the LASSO estimate $\hat \bbeta$ will include the union of the supports of $\bbeta_j$, provided that sufficiently strong minimal signal (or sufficiently large sample size) and an irrepresantable condition are present.  How to apply the LASSO algorithm for support recovery under the more general class of multi-index models warrants further research. 

A note on the choice of the tuning parameter $\lambda$ in practice is also in order. According to Remark \ref{signal:strength:remark} and optimal choice of $\lambda$ may depend on the unknowable parameter $\xi^2$. A procedure which we found to work well in practice is based on simple cross validation. To find a good tuning parameter $\lambda$, from a grid of $\ell$ values of $\lambda$: $\{\lambda_1,\ldots, \lambda_\ell\}$ which are of $\sqrt{\frac{\log p}{n}}$ magnitude, we recommend using $K$-fold cross validation, and setting $\lambda$ to the value minimizing the average least squares loss across (i.e. mean squared-error) the $K$ folds. Notice that according to Lemma \ref{proplemma}, this criteria is very sensible. An important question is whether one can arrive at a procedure with theoretical guarantees for $\lambda$ selection, and we hope to address this problem in our future work.

Finally, under a SIM and proper distributional assumptions on $\bX$, one may also recover $\bbeta_0$ proportionally using other convex loss functions. For example, when $Y$ is binary, the logistic log-likelihood loss may be more efficient than the $L_2$ loss. Thus, it is of interest to investigate the support recovery properties of the LASSO (or more general penalization procedures) with other convex losses --- such as the logistic/hinge losses, which could be less susceptible to outliers.

\section{Acknowledgments}
The authors would like to thank the Associate Editor and anonymous reviewers for their insightful remarks which led to the improvement of this manuscript. We are also grateful to Professor Noureddine El Karoui, who among other observations, brought to our attention that outcome transformations may be performed to facilitate the usage of LASSO even when $c_0 = 0$ for the original outcome $Y$. This research was partially supported by Research Grants NSF DMS1208771, NIH R01 GM113242-01, NIH U54 HG007963 and NIH RO1 HL089778.

\newpage

\appendix

\section{Auxiliary Lemmas} 

In this section, for convenience of the reader, we state couple of lemmas that we use often in our analysis.

\begin{lemma}[Corollary 5.35 \cite{vershynin2010introduction}] \label{cor535versh} Let $\Ab_{n \times s}$ matrix whose entries are i.i.d. standard normal random variables. Then for every $t \geq 0$, with probability at least $1 - 2 \exp(-t^2/2)$ one has:
$$
\sqrt{n} - \sqrt{s} - t \leq s_{\min}(\Ab) \leq s_{\max}(\Ab) \leq \sqrt{n} + \sqrt{s} + t,
$$
where $s_{\min}(\Ab)$ and $s_{\max}(\Ab)$ are the smallest and largest singular values of $\Ab$ correspondingly.
\end{lemma}

\begin{lemma} \label{lemma5wain:mod} Consider a \underline{fixed} nonzero vector $\zb \in \RR^s$ and a random matrix $\Ab_{n \times s}$, whose entries are i.i.d. standard normal random variables. If $p,s,n$ are such that $s \geq 2$, $\frac{s}{n} \leq \frac{1}{64}$ and $\frac{\log p}{n -s +1} \leq \frac{1}{32}$, there are positive absolute constants $C_1$ and $C_2$ satisfying:
$$
\p\left(\|[(n^{-1}\Ab\T \Ab)^{-1} - \eye_{s \times s}] \zb \|_{\infty} \geq C_1\frac{s}{n}\|\zb\|_{\infty} + C_2\|\zb\|_2 \sqrt{\frac{\log p}{n}} \right) \leq 4 p^{-1}.%4 \exp(-C_2 (???)),
$$
\end{lemma}

\begin{proof}[Proof of Lemma \ref{lemma5wain:mod}] This Lemma is a generalization/modification of Lemma 5 of \cite{wainwright2009sharp}, allowing us to make usage of the $L_2$ norm $\|\zb\|_2$ to obtain more precise bounds. For self-content we spell out the full details of the proof. Using the spectral theorem, decompose the matrix $(n^{-1}\Ab\T \Ab)^{-1} - \eye = \Vb \Db \Vb\T$, where $\Db$ is a diagonal matrix and $\Vb$ is an independent of $\Db$ unitary matrix. Define the random variables:
$$
U_i = \eb_i\T \Vb \Db \Vb\T \zb = z_i \vb_i\T\Db \vb_i + \vb_i\T \Db \sum_{j \neq i} z_j \vb_j,
$$
where $\vb_i\T$ is the $i$\textsuperscript{th} row vector of the matrix $\Vb$. To bound $\max_i |U_i|$ we deal with these two terms in turn. First notice that $\vb_i\T\Db \vb_i$ is the $i$\textsuperscript{th} diagonal entry of the matrix $(n^{-1}\Ab\T \Ab)^{-1} - \eye$. By the assumption $(\Ab\T \Ab)^{-1} \sim \cW^{-1}(\eye_{s \times s}, n)$ where $\cW^{-1}$ is an inverse Wishart distribution. By the properties of the inverse Wishart distribution we conclude that $\vb_i\T\Db \vb_i \sim n \chi^{-2}(n - s + 1) - 1$, where $\chi^{-2}$ is the inverse $\chi^2$ distribution. Hence using Lemma 1 of \cite{laurent2000adaptive} and the union bound we have:
$$
\PP(\max_{i} |1 - ((\vb_i\T\Db \vb_i + 1) (n - s +1)/n)^{-1}| \geq 2 \sqrt{y} + 2y) \leq 2 s\exp(-(n - s +1)y).
$$
Selecting $y = \frac{2\log p}{n - s + 1}$ bounds the above probability by $2s/p^2 \leq 2/p$. For values of $y < 1/32$ we can lump $2\sqrt{y} + 2y < \Big(2 + \frac{\sqrt{2}}{4}\Big)\sqrt{y}$. Thus inverting the inequality inside the probability we conclude that for each $i \in [s]$:
$$
 \frac{n}{(n-s+1)(1 +  \Big(2 + \frac{\sqrt{2}}{4}\Big)\sqrt{\frac{2\log p}{n - s + 1}})}  \leq \vb_i\T\Db \vb_i + 1\leq \frac{n}{(n-s + 1)(1 -  \Big(2 + \frac{\sqrt{2}}{4}\Big)\sqrt{\frac{2\log p}{n - s + 1}})} 
 $$
It is simple to see that when $\frac{\log (p)}{n - s + 1} \leq \frac{1}{32}$ the above implies:
\begin{align}\label{vDv:bound}
\max_{i}  |z_i \vb_i\T\Db \vb_i| <  \|\zb\|_{\infty} \tilde c_1 \frac{s}{n-s} + \|\zb\|_2\tilde c_2 \sqrt{\frac{\log p}{n - s + 1}},
\end{align}
where $\tilde c_1 = (1 - (1/2 + \sqrt{2}/16))^{-1}$ and $\tilde c_2 = (2 + \sqrt{2}/4)\tilde c_1$. Next we show that the function $F(\vb_i) = \vb_i\T \Db \sum_{j \neq i} z_j \vb_j$ is Lipschitz with a constant $8\sqrt{s/n} \|\zb\|_2$. We have:
$$
\|\nabla F\|_2 \leq \|\Db\|_{2,2} \Big\|\sum_{j \neq i} z_j \vb_j\Big\|_{2} \leq 9 \sqrt{\frac{s}{n}}\|\zb\|_2,
$$
with the last inequality holding with probability at least $1 - 2 \exp(-s/2)$ when $\frac{s}{n} \leq \frac{1}{64}$. We used that $\vb_j$ are orthonormal, and we bounded the maximum eigenvalue of $\Db$, which follows just as in the proof of Lemma \ref{sphereconclemma} so we omit the details. Since the variables $\{\vb_j\}_{j \neq i}$ are uniformly distributed on a $(s-1)$-dimensional sphere, the proof is completed by invoking the concentration of Lipschitz functions on the sphere to bound $\max_{i \in [s]} |F(\vb_i)|$:
$$
\p(\max_{i \in [s]} |F(\vb_i)| \geq t) \leq 2s\exp\left(- \tilde c (s-1) \frac{t^2}{81 \frac{s}{n}\|\zb\|_2^2}\right),
$$
for an absolute constant $\tilde c$. Under the assumption $s \geq 2$, we can select $t = 18 \|\zb\|_2 \sqrt{\frac{\log p}{\tilde c n}}$ and taking into account that $p > s$ completes the proof, after noticing that we can absorb the second term of (\ref{vDv:bound}) to the above expression. 
%The details are as in Lemma \ref{sphereconclemma} so we omit them.  
 % Hence for all $j$ we have $\ub_j\T\Db \ub_j \stackrel{d}{=} \frac{n}{\sum_{i =1 }^{n - s +1} \chi^2_{ij}} - 1$, where each $\chi^2_{ij} \sim \chi^2(1)$, and $\chi^2_{ij}$ are independent for $i \in [n - s + 1]$ for a fixed $j$.
\end{proof}

\section{Preliminary Results}\label{prelim:res:appendix}

\begin{proof}[Proof of Lemma \ref{proplemma}] First let $\bSigma = \mathbb{I}_{p\times p}$ (hence by assumption $\|\bbeta_0\|_2 = 1$) and take any $\bb \perp \bbeta_0$. Note that by the linearity of expectation:
$$
\e[\bX\T\bbeta_0\bX\T \bb | \bX\T \bbeta_0] = c_{\bb} (\bX\T \bbeta_0)^2.
$$
Taking another expectation above we have $\e[\bX\T\bbeta_0\bX\T \bb ] = c_{\bb} \e[(\bX\T \bbeta_0)^2]$. However
$$
\e[\bX\T\bbeta_0\bX\T \bb ] =  \bb\T\bbeta_0 = 0,
$$
and hence $c_{\bb} = 0$. Thus if $\bb \perp \bbeta_0$, $\e[\bX\T \bb | \bX\T \bbeta_0] = 0$. Next, for any $\bb \perp \bbeta_0$ we have:
$$
\e[Y\bX\T \bb] = \e[ \e[Y\bX\T \bb | \bX\T \bbeta_0] ] = \e[ \e[Y | \bX\T \bbeta_0] \e[\bX\T \bb | \bX\T \bbeta_0]] = 0. 
$$
Hence $\e[Y\bX] \varpropto \bbeta_0$. Finally, a projection on $\bbeta_0$ yields
$$
c_0 \|\bbeta_0\|_2^2 = \e[Y\bX\T\bbeta_0].
$$ 
This completes the proof in the case when $\bSigma = \mathbb{I}_{p\times p}$. For the more general case observe that $Y = f(\bX\T\bbeta_0, \varepsilon) = f(\bX\T\bSigma^{-1/2} \bSigma^{1/2}\bbeta_0, \varepsilon)$, and thus by what we just saw we have:
$$
\EE[Y \bSigma^{-1/2}\bX] = c_0 \bSigma^{1/2} \bbeta_0,
$$
which becomes what we wanted to show after multiplying by $\bSigma^{-1/2}$ on the left.
\end{proof}

\section{Lower Bound}\label{lower:bound:appendix}

%\begin{proof}[Proof of Proposition \ref{lower:bound:prop}]
%\end{proof}
For two probability measures $P$ and $Q$, which are absolutely continuous with respect to a third probability measure $\mu$ (i.e. $P, Q \ll \mu$), define their KL divergence by $\DKL(P \| Q) = \int p \log \frac{p}{q} d\mu$, where $p = \frac{dP}{d\mu}$, $q = \frac{dQ}{d\mu}$.

\begin{lemma}\label{lowerboundthm:new} Assume conditions required in Proposition \ref{lower:bound:prop:new}. In addition, let for any fixed $u, v \in \mathbb{R}$ and some positive constant $\Xi$, $f(u,\varepsilon)$ and $f(v,\varepsilon)$ satisfy
\begin{align}
\DKL(p(f(u, \varepsilon)) \| p(f(v, \varepsilon))) \leq \exp(\Xi(u - v)^2) - 1, \label{KLdivcond} 
\end{align}
Then if
$$
n_{p,s} < \frac{1}{8 \Xi}, \quad \mbox{and} \quad s \geq 8 \Xi,
$$
any algorithm recovering the support of $\bbeta_0$ under (\ref{SIM}) will have errors with probability at least $\frac{1}{2}$ asymptotically.
\end{lemma}

\begin{proof}[Proof of Lemma \ref{lowerboundthm:new}] We start by constructing a set of $p - s$ vectors $\bB = \{\bbeta_1, \ldots, \bbeta_{p-s}\}$, belonging to the parameter space $\{\bbeta \in \RR^p: \bbeta\T\bSigma\bbeta = 1,  \|\bbeta\|_0 = s, \frac{\bbeta^{\min}}{\|\bbeta\|_{\infty}} \geq c_{\bSigma}\},$  for a sufficiently small (to be chosen) $c_{\bSigma} > 0$, such that $\var(\bX\T(\bbeta_k - \bbeta_j)) \leq \frac{4}{s}$ for all $k,j \in [p-s]$. Once this set is constructed we will use standard Fano type of argument to finish the proof. 

Without loss of generality let us assume that $S_0 = [s]$. To construct the set $\bB$, first focus on the sub-matrix $\bSigma_{S_0,S_0}$. We take the $s$ dimensional vector $\bgamma = a(1/\sqrt{s}, \ldots, 1/\sqrt{s}, 0)\T$ where $a > 0$ is selected so that $\bgamma\T\bSigma_{S_0,S_0}\bgamma = \frac{s-1}{s}$. Since $\bSigma_{S_0,S_0}$ is assumed to have bounded eigenvalues we know that such $a$ indeed exists, and can be chosen in the interval $a \in [\frac{1}{2}\frac{1}{\sqrt{\lambda_{\max}}}, \frac{1}{\sqrt{\lambda_{\min}}}]$ for $s \geq 2$. Next to construct $\bbeta_k$ we use:
$$\beta_{kr} = \gamma_r\mathds{1}(r \in S_0) + b_k\mathds{1}(r = k + s)/\sqrt{s},$$
where $b_k$ is chosen so that $\bbeta_k\T \bSigma \bbeta_k = 1$. Below we argue that such $b_k$ indeed exists. By H\"{o}lder's inequality we have $\|\bSigma_{S_0,S_0} \bgamma\|_{\infty} \leq a\|\bSigma_{S_0,S_0}\|_{\infty, \infty} /\sqrt{s} \leq aR/\sqrt{s}$. Hence using Assumption \ref{irrepcond} we have:
$$
\|\bSigma_{S_0^c,S_0}\bSigma^{-1}_{S_0,S_0}\bSigma_{S_0,S_0} \bgamma\|_{\infty} \leq \|\bSigma_{S_0^c,S_0}\bSigma^{-1}_{S_0,S_0}\|_{\infty,\infty} aR/\sqrt{s} \leq (1 - \kappa)aR/\sqrt{s}.
$$
Note that due to the last inequality, for any $k \in [p-s]$ we have:
$$
\bbeta_k\T \bSigma \bbeta_k = \bgamma\T \bSigma_{S_0,S_0}\bgamma + 2\bSigma_{\{k + s\},S_0}\bgamma\frac{b_k}{\sqrt{s}} + \frac{b_k^2\Sigma_{k + s,k + s}}{s} \leq \frac{s-1}{s} + 2\frac{(1 - \kappa)a|b_k|R}{s} + \frac{b_k^2 \Sigma_{k + s,k + s}}{s},
$$
where we remind the reader that $k + s \in S_0^c$. We chose $b_k$ such that $\sign(b_k) = \sign(\bSigma_{\{k + s\},S_0}\bgamma)$. Hence we also have:
$$
\frac{s-1}{s} + \frac{b_k^2\Sigma_{k + s, k + s}}{s}  \leq \bbeta_k\T \bSigma \bbeta_k.
$$
Combining the last two inequalities we conclude that there exists:
$$|b_k| \in \Big[\frac{\sqrt{(1 - \kappa)^2 a^2 R^2 + \Sigma_{k + s,k + s}} - (1 - \kappa)a R}{\Sigma_{k + s,k + s}} ,\Sigma^{-1/2}_{k + s,k + s}\Big],$$
with the desired properties. One can easily check that when:
\begin{align}\label{c_sigma:def}
c_{\bSigma} \leq \min\Bigg(\frac{1}{2} \sqrt{\frac{d}{\lambda_{\max}}}, \frac{\sqrt{(1 - \kappa)^2R^2 + d\lambda_{\min}} - (1-\kappa)R}{D}\Bigg),
\end{align}
we have that $\min_{j \in [p-s]}\frac{\bbeta_j^{\min}}{\|\bbeta_j\|_{\infty}} \geq c_{\bSigma}$, and in addition as we promised we have:
$$
\var(\bX\T(\bbeta_k - \bbeta_j)) = \frac{\Sigma_{k + s,k + s} b_k^2 - 2 \Sigma_{k + s, j + s}b_kb_j + \Sigma_{j + s,j + s} b_j^2}{s} \leq \frac{4}{s}.
$$

Next, let $\Jsc$ be a uniform distribution on $\bB$. Under the $0-1$ loss the risk equals the probability of error:
\begin{align}\frac{1}{p - s} \sum_{j \in [p-s]}P_{\smbbeta_j} \{\widehat{S}\neq S({\bbeta_j})\}, \label{errorsum}
\end{align}
where by $P_{\smbbeta_j}$ we are measuring the probability under a dataset generated with $\bbeta_j$, and $\hat S$ is an estimate of the true support produced by any (possibly randomized) algorithm. By Fano's inequality that:
\begin{align}\label{fanosineq}
\p(\mbox{error}) \geq 1 - \frac{\cI(\Jsc; \Dsc) + \log(2)}{\log|{\bB}|},
\end{align}
where $\cI(\Jsc; \Dsc)$ is the mutual information between the sample $\Jsc$ and the sample $\Dsc$. Note now that for the mutual information we have 
\begin{align*}
\cI(\Jsc; \Dsc) & = \cI(\Jsc; [\{f(\bX_j\T\bbeta_0, \varepsilon_j) ,\bX_j\},j=1,\ldots,n])  \\
& \leq n \cH[\{f(\bX\T\bbeta_0, \varepsilon),\bX\}] - n \cH[\{f(\bX\T\bbeta_0, \varepsilon),\bX\} | \Jsc] \\
& \leq n \max_{k, j \in [p-s]}\DKL\big( p[\{f(\bX\T\bbeta_{k}, \varepsilon), \bX\}] \| p[\{f(\bX\T\bbeta_{j}, \varepsilon), \bX\}]\big),
\end{align*}
where $\cH(\cdot)$ denotes the marginal entropy, $\cH(\cdot \mid \cdot)$ denotes the conditional entropy, $\DKL$ denotes the KL divergence,
the first  inequality follows from the chain inequality of entropy and the second inequality follows from a standard bound. 
Since the KL divergence is invariant under change of variables, we let $U_k = \bX\T \bbeta_k$ and $\bW_{kj} = \Pb_{\bSigma, \{\bbeta_k, \bbeta_{j}\}^\perp}\bX$, where $\Pb_{\bSigma, \{\bbeta_k, \bbeta_{j}\}^\perp} \in \RR^{(p - 2) \times p}$ is  chosen such that $\Pb_{\bSigma, \{\bbeta_k, \bbeta_{j}\}^\perp}\bSigma \bbeta_k =0$ and $\Pb_{\bSigma, \{\bbeta_k, \bbeta_{j}\}^\perp}\bSigma \bbeta_j =0$. Noting that $\bW_{kj}$ is independent of $U_k, U_j, \varepsilon$, $U_k - U_j \sim \cN(0, V)$ with $V \le 4/s$, and applying assumption (\ref{KLdivcond}), we have  
\begin{align*}
n^{-1} \cI(\Jsc; \Dsc) & \leq \max_{k, j \in [p-s]}\DKL \big( p[\{f(U_k, \varepsilon), U_k, U_j\}]  \| p[ \{f(U_j, \varepsilon), U_k, U_j\}] \big) \\
& \leq \max_{k, j \in [p-s]} \e \exp(\Xi(U_k - U_j)^2) - 1  \leq \sqrt{\frac{s}{ s - 8 \Xi}} - 1 \leq \frac{8\Xi}{s}.
\end{align*}
where the second inequality can be obtained by conditioning on $U_k, U_j$, and we assume that the value of $s$ is large enough so that $s\geq 16\Xi$. We conclude that:
$$
\cI(\Jsc; \Dsc) \leq \frac{8\Xi n}{s}.
$$
Consequently by (\ref{fanosineq}) if $n_{p,s} < 1/(8\Xi)$ we will have errors with probability at least $\frac{1}{2}$, asymptotically. This is what we wanted to show. %Finally observe that by splitting the set  $\mathbb{S}$ into sets such as $\tilde{\mathbb{S}}$ one can see by exactly the same argument that (\ref{errorsum}) has to be at least $\frac{1}{2}$. This is what we wanted to show.
%To finish the conclusion, note that the sum (\ref{errorsum}), can be split into ${p \choose s - 1}$ terms, by the following operation: 
%
%\begin{enumerate}
%\item Repeat each set in $\mathbb{S}$ --- $s$ times, and denote this superset by $s \times \mathbb{S}$
%\item For each $S$ of the ${p \choose s - 1}$, subsets of $[p]$ with $s-1$ elements, collect $p-s + 1$ distinct elements of $s \times \mathbb{S}$ containing $S$
%\item Apply the $\frac{1}{2}$ error bound obtained from above to this local sum.
%\end{enumerate}
%
%In the end we get that the probability of error by selecting $S \subset \mathbb{S}$ uniformly is at least: $\frac{1}{s}\frac{{p \choose s-1}}{{p \choose s}} (p - s + 1) \frac{1}{2} = \frac{1}{2}$.
\end{proof}

\begin{proof}[Proof of Proposition \ref{lower:bound:prop:new}]  Note that all moments of the random variable $\varepsilon$ exist. Next we verify that condition (\ref{KLdivcond}) of Lemma \ref{lowerboundthm:new} holds in this setup. Since $G$ is 1-1 and KL divergence is invariant under changes of variables WLOG we can assume our model is simply $Y = h(\bX\T \bbeta_0) + \varepsilon$ or in other words $f(u, \varepsilon) = h(u) + \varepsilon$. This is a location family for $u \in \RR$ and thus the normalizing constant of the densities will stay the same regardless of the value of $u$. Direct calculation yields:
\begin{align*}
\DKL[p\{f(u, \varepsilon)\} \| p\{f(v, \varepsilon)\}] & = \e [P((\xi + h(u) - h(v))^2) - P(\xi^2)] = \widetilde{P}((h(u) - h(v))^2),
\end{align*}
where $\xi$ has a density $p_{\xi}(x) \varpropto \exp(-P(x^2))$, and $\widetilde{P}$ is another non-zero polynomial with nonnegative coefficients, with $\widetilde{P}(0) = 0$ of the same degree as $P$. The last observation follows from the fact that all odd moments of $\xi$ are $0$, since $\xi$ is a symmetric about $0$ distribution. Since $h$ is $L$-Lipschitz we conclude that:
\begin{align*}
\DKL[p\{f(u, \varepsilon)\}\| p\{f(v, \varepsilon)\}] & \leq \widetilde{P}(L^2(u - v)^2).
\end{align*}
The last can be clearly dominated by $\exp(C(u-v)^2) - 1$ for a large enough constant $C$.
\end{proof}

\section{Covariance Thresholding}\label{cov:tresh:app}

\begin{proof}[Proof of Proposition \ref{covthresh}]
%Denote by $S_{0} := S(\bbeta_0)$ the support of the vector $\bbeta_0$. We have $|S_{0}| = s$. 
Using the fact that for any two random variables $R,T$ we have $\|RT\|_{\psi_1} \leq 2 \|R\|_{\psi_2} \|T\|_{\psi_2}$ we can conclude that the random vector $Y \bX$ is coordinate-wise sub-exponentially distributed since $\sup_{j \in [p]}\|Y X_j\|_{\psi_1} \leq \cK := 2 K_Y K$. An application of Proposition 5.16 of \cite{vershynin2010introduction} and the union bound then gives us that:
$$
\p\left(\left\|\frac{1}{n} \sum_{i = 1}^n Y_i \bX_i - \e[Y\bX]\right\|_{\infty} \geq t\right) \leq 2p \exp\left[- \tilde{c} \min \left(\frac{n t^2}{\cK^2}, \frac{n t}{\cK} \right)\right],
$$
where $\tilde{c} > 0$ is some absolute constant. This inequality then implies: 
$$\sup_{j \in [p]}\left|\frac{1}{n} \sum_{i = 1}^n Y_i X_{ij} - \e(YX_j)\right|\leq \cK\sqrt{ \frac{2 \log p}{\tilde{c} n}},$$ 
with probability at least $1 - 2p^{-1}$ for values of $n,p$ such that $\frac{\log p}{n} \leq \frac{\tilde{c}}{2}$. Note that the inequality in the preceding display implies that if:
$$
\frac{|c_0|}{\sqrt{s}} > R\cK \sqrt{ \frac{2 \log p}{\tilde{c} n}},
$$
for any $R > 2$ there will be a gap in the absolute values of the coefficients of $U_j = |n^{-1} \sum_{i = 1}^n Y_i X_{ij} |$ for $j \in S_{0}$ and $j \not\in S_{0}$. The latter happens because:
$$
\frac{|c_0|}{\sqrt{s}} - \cK\sqrt{ \frac{2 \log p}{\tilde{c} n}} \geq (R - 1)\cK \sqrt{ \frac{2 \log p}{\tilde{c} n}} >  \cK\sqrt{ \frac{2 \log p}{\tilde{c} n}}.
$$
This also shows that the coefficients will achieve the correct sign. Thus, as long as $\frac{n}{s \log p} \geq \Upsilon$, for $\Upsilon = \frac{2 R^2 \cK^2}{c_0^2 \tilde c^2}$ signed support recovery happens with asymptotic probability $1$. Under our assumption the latter is implied by $n_{p,s} > \Upsilon/\iota$ which completes the proof. 
\end{proof}

\begin{proof}[Proof of Proposition \ref{propnormal}] We follow the same steps as the proof of Proposition \ref{covthresh}. We will use the following Lemma which we justify after the proof:

\begin{lemma} \label{subgaussboundnormal} Let us observe $n$ data points from the model described in Proposition \ref{propnormal} with $\bbeta$ being an arbitrary unit vector. Then with probability at least $1 - \frac{\eta}{n \sigma^4} - \frac{\gamma}{\log p} - \frac{2}{p}$ the following event holds:
$$
\left\|\frac{1}{n} \sum_{i = 1}^n Y_i \bX_i - \EE(Y \bX) \right\|_{\infty} \leq \Big(\|\bbeta_0\|_{\infty} +  2\sqrt{2\sigma^2}\Big) \sqrt{\frac{\log p}{n}}.
$$
\end{lemma}
Using the fact that in our case $\EE(Y \bX) = c_0 \bbeta_0$, and that $\|\bbeta_0\|_{\infty} = \frac{1}{\sqrt{s}}$, we have that if:
$$
%\sqrt{n} > 2|c_0| n^{\alpha} + 4 |c_0| \sqrt{2\sigma^2s \log p},
\Big(\|\bbeta_0\|_{\infty} +  2\sqrt{2\sigma^2}\Big) \sqrt{\frac{\log p}{n}} \leq (1 + 2\sqrt{2}\sigma)\frac{1}{\sqrt{s}}\sqrt{\frac{s\log p}{n}} < \frac{|c_0|}{2}\frac{1}{\sqrt{s}}
$$
there will be a gap between the coefficients corresponding to $j \in S_{0} := S(\bbeta_0)$ and $j \not \in S_{0}$. Note that the last inequality holds if $\frac{s \log p}{n} < \frac{c_0^2}{4 (1 + 2 \sqrt{2}\sigma)^2}$.%since $\alpha < 1/2$, this condition holds for large values of $n$ as long as $\sqrt{n} > 8 |c_0| \sqrt{2\sigma^2s \log p}$ e.g. This completes the proof.
\end{proof}

\begin{remark} The slow convergence in probability rate $(\log p)^{-1}$ observed in Lemma \ref{subgaussboundnormal} is due to the fact that we are not requiring that $Y$ is sub-Gaussian. If we do require it, the convergence rate of the probability can be seen to reduce to the usual $p^{-1}$ level.
\end{remark}

\begin{proof}[Proof of Lemma \ref{subgaussboundnormal}]
Note that, sub-exponential concentration bounds do not apply in this case. However, observe that by the properties of the multivariate normal distribution the random variable $(\Ibb - \bbeta_0\bbeta_0\T)\bX$ is independent of $\bX\T \bbeta_0$ and hence is independent of $Y$. Furthermore it is clear that the random variable $Y(\Ibb - \bbeta_0\bbeta_0\T)\bX$ has mean $0$. Note that conditional on $Y_{i}, i \in [n]$ we have that $\frac{1}{n} \sum_{i = 1}^n Y_i (\Ibb - \bbeta_0\bbeta_0\T)\bX_{i} \sim \cN(0, n^{-2}\sum_{i = 1}^n Y^2_i (\Ibb - \bbeta_0\bbeta_0\T))$. Thus by a standard Gaussian tail bound:
$$
\p\left(\left\|\frac{1}{n} \sum_{i = 1}^n Y_i (\Ibb - \bbeta_0\bbeta_0\T)\bX_i\right\|_{\infty} \geq t \bigg \vert \bY \right) \leq 2p \exp\left[- \frac{n t^2}{2 \overline{Y^2}}\right],
$$
where $\overline{Y^2} = n^{-1}\sum_{i = 1}^n Y_i^2$, and we used that $\|\eye - \bbeta_0\bbeta_0\T\|_{2,2} \leq 1$. By Chebyshev's inequality $\p(|\overline{Y^2} - \sigma^2| \geq r ) \leq \frac{\eta}{n r^2}$. Hence selecting $r = \sigma^2$ will keep the above probability going to 1 at rate $\frac{1}{n}$ and moreover for large $n$ we have $\overline{Y^2} \leq 2\sigma^2$. Using this bound in the tail bound above yields that for a choice of $t = 2\sqrt{2\sigma^2\frac{\log p}{n}}$ the tail bound will go to 0 at rate $2p^{-1}$ as claimed.

Next consider controlling:
$$
\p\left(\left\|\frac{1}{n} \sum_{i = 1}^n Y_i \bbeta_0\bbeta_0\T\bX_i - c_0\bbeta_0 \right\|_{\infty} \geq t \right) = \p\left(\left|\frac{1}{n} \sum_{i = 1}^n Y_i \bX_i\T\bbeta_0 - c_0 \right|\geq t/\|\bbeta_0\|_{\infty} \right) ,
$$
where recall that $\EE(Y\bX) = c_0\bbeta_0$, and $c_0$ is defined in the main text. Applying Chebyshev's inequality once again we get that $t = \|\bbeta_0\|_{\infty}\sqrt{\frac{\log p}{n}}$ suffices to keep the above probability going to $0$. By the triangle inequality we conclude that, with probability going to 1:
$$
\left\|\frac{1}{n} \sum_{i = 1}^n Y_i \bX_i - \EE(Y \bX)\right\|_{\infty} \leq  \|\bbeta_0\|_{\infty}\sqrt{\frac{\log p}{n}} +  2\sqrt{2\sigma^2 \frac{\log p}{n}}.
$$
This is what we claimed.
\end{proof}

\section{LASSO Support Recovery}

\begin{proof}[Proof of Lemma \ref{newlemma4}] Note that since $\Pb_{\Xbb_{,S_0}^\perp}$ is an orthogonal projection matrix it contracts length and hence:
$$
\left\|\Pb_{\Xbb_{,S_0}^\perp} \left(\frac{\bfw}{\lambda n}\right)\right\|_2^2 \leq \frac{\|\bfw\|^2_2}{\lambda^2 n^2}.
$$
Next observe that $\bfw = \bY - c_0 \bfX\bbeta_0$ is a vector with non-zero mean. However, by Chebyshev's inequality we have:
$$
\p\left(\left|\frac{\|\bfw\|^2_2}{n} - \xi^2\right| \geq t \right) \leq \frac{\theta^2}{n t^2}.
$$
Then setting $t = 1$ brings the above probability to $0$ at a rate $\frac{\theta^2}{n}$. Next:
\begin{align*}
n^{-1} \check \zb_{S_0}\T (n^{-1} \Xbb_{,S_0}\T \Xbb_{,S_0})^{-1} \check \zb_{S_0} \leq \frac{1}{\lambda_{\min}(1 - 2\sqrt{\frac{s}{n}})^2} \frac{\|\check \zb_{S_0}\|_2^2}{n} \leq \frac{1}{\lambda_{\min}(1 - 2\sqrt{\frac{s}{n}})^2}\frac{s}{n} ,
\end{align*}
with probability at least $1 - 2\exp(-s/2)$, where we used Lemma \ref{cor535versh}. This completes the proof.
\end{proof}

\begin{proof}[Proof of Lemma \ref{superimportantlemma}]
First, we note the following decomposition:
$$
[\Xbb_{,S_0}\T \Xbb_{,S_0}]^{-1} \Xbb_{,S_0}\T \bY - c_0 \bbeta_{0S_0} = (n[\Xbb_{,S_0}\T \Xbb_{,S_0}]^{-1} - \eye) n^{-1}\Xbb_{,S_0}\T \bY+ (n^{-1}\Xbb_{,S_0}\T \bY - c_0 \bbeta_{0S_0}).
$$
Note that the second term is mean $0$. Applying Lemma \ref{subgaussboundnormal} gives us a bound on the second term. We next move on to consider the first term. 
%$$
%\p(\|(n^{-1}X\T Y - c_0 \bbeta)\|_{\infty} \geq t) \leq 2 s \exp\left[-C \min\left(\frac{t^2 n}{4K^4}, \frac{tn}{2 K^2}\right)\right],
%$$
%where $C > 0$ is an absolute constant. We therefore have the freedom to select $t = \frac{\sqrt{2}4 K^2}{\sqrt{C}}\sqrt{\frac{{\log s}}{n}}$, to bound the probability by $\frac{2}{s}$. 

Consider a ``symmetrization'' transformation of the predictor matrix $\widetilde \Xbb_{,S_0}\T = (\eye - \bbeta_{0S_0} \bbeta_{0S_0}\T) \Xbb_{,S_0}\T + \bbeta_{0S_0} \bbeta_{0S_0}\T \Xbb_{,S_0}^{*\intercal}$, where $[\Xbb_{,S_0}^*]_{n \times s}$ is an i.i.d. copy of $\Xbb_{,S_0}$, or in other words the columns of $\Xbb_{,S_0}^*$: $\bX_{j}^* \sim \cN(0, \eye_{n \times n}), j = 1,\ldots, s$ and are independent of $\Xbb_{,S_0}$. Note that in doing this construction, we guarantee that $\widetilde \Xbb_{,S_0}$ is independent of $\Xbb_{,S_0}\T \bbeta_{0S_0}$. Now we further decompose the first term as follows:
\begin{align*}
(n[\Xbb_{,S_0}\T \Xbb_{,S_0}]^{-1} - \eye) n^{-1}\Xbb_{,S_0}\T \bY & = \underbrace{(n[\Xbb_{,S_0}\T \Xbb_{,S_0}]^{-1} - \eye)\bbeta_{0S_0} \bbeta_{0S_0}\T n^{-1}\Xbb_{,S_0}\T \bY}_{\bI_1} \\
&+ \underbrace{n([\Xbb_{,S_0}\T \Xbb_{,S_0}]^{-1} - [\widetilde \Xbb_{,S_0}\T \widetilde \Xbb_{,S_0}]^{-1})(\eye - \bbeta_{0S_0} \bbeta_{0S_0}\T )n^{-1}\Xbb_{,S_0}\T \bY}_{\bI_2}\\
& + \underbrace{(n[\widetilde \Xbb_{,S_0}\T \widetilde \Xbb_{,S_0}]^{-1} - \eye)n^{-1} \widetilde \Xbb_{,S_0}\T \bY}_{\bI_3}\\ & - \underbrace{(n[\widetilde \Xbb_{,S_0}\T \widetilde \Xbb_{,S_0}]^{-1} - \eye)\bbeta_{0S_0} \bbeta_{0S_0}\T n^{-1} \Xbb_{,S_0}^{* \intercal} \bY}_{\bI_4}.
\end{align*}
We next deal with each of these terms separately. For the first and last terms we can directly apply Lemma \ref{lemma5wain:mod}. Under the same event as in Lemma \ref{inftyinftynorm}, taking into account that $\|\bbeta_{0S_0}\|_2 = 1$ we have that $\|([n^{-1}\Xbb_{,S_0}\T \Xbb_{,S_0}]^{-1} - \eye) \bbeta_{0S_0}\|_{\infty} \leq C_1\frac{s}{n}\|\bbeta_{0S_0}\|_{\infty} + C_2 \sqrt{\frac{\log p}{n}}$ and $\|([n^{-1} \widetilde \Xbb_{,S_0}\T \widetilde \Xbb_{,S_0}]^{-1} - \eye) \bbeta_{0S_0}\|_{\infty} \leq C_1\frac{s}{n}\|\bbeta_{0S_0}\|_{\infty} + C_2 \sqrt{\frac{\log p}{n}}$. Furthermore, $\bbeta_{0S_0}\T \Xbb_{,S_0}\T \bY/n$ is a mean $c_0$ random variable. Just as in the proof of Lemma \ref{subgaussboundnormal} by Chebyshev's inequality we have that with probability at least $1 - \frac{\gamma}{\log p}$ we have $|\bbeta_{0S_0}\T \Xbb_{,S_0}\T \bY/n| \leq |c_0| + \sqrt{\frac{\log p}{n}}$. 

Furthermore, notice that $n^{-1}\bbeta_{0S_0}\T \Xbb_{,S_0}^{*\intercal} \bY$ is a mean $0$ random variable. Conditionally on $\bY$ it has a $\cN(0, n^{-2} \sum Y_i^2)$ distribution. With exactly the same argument as in the proof of Lemma \ref{subgaussboundnormal} we conclude that with probability at least $1 - \frac{\eta}{n \sigma^4} - \frac{2}{p}$:
$$
|n^{-1}\bbeta_{0S_0}\T \Xbb_{,S_0}^{*\intercal} \bY| \leq 2 \sqrt{\sigma^2 \frac{\log p}{n}}.
$$
Hence, combining the above results we obtain:
%\begin{align*}
%\|(n[\Xbb\T \Xbb]^{-1} - \eye) n^{-1}\Xbb\T \bY\|_{\infty} & \leq C_1 \|\bbeta\|_{\infty}\left(|c_0| + n^{-1/2 + \alpha} +  2 \sqrt{2\sigma^2  \frac{\log n}{n}}\right) + \|\bI_2\|_{\infty} +  \|\bI_3\|_{\infty}.
%\end{align*}
\begin{align}\label{i1:i4:bound}
\|\bI_1\|_{\infty} + \|\bI_4\|_{\infty}& \leq \Big(C_1\frac{s}{n}\|\bbeta_{0S_0}\|_{\infty} + C_2 \sqrt{\frac{\log p}{n}}\Big)\left(|c_0| + \sqrt{\frac{\log p}{n}} +  2 \sqrt{\sigma^2  \frac{\log p}{n}}\right).
\end{align}
To deal with the term $\bI_2$ first note that by H\"{o}lder's inequality we have:
\begin{align}\label{i2:bound}
\|\bI_2\|_{\infty} \leq  \|n([\Xbb_{,S_0}\T \Xbb_{,S_0}]^{-1} - [\widetilde \Xbb_{,S_0}\T \widetilde \Xbb_{,S_0}]^{-1})\|_{\infty,\infty}\|(\eye - \bbeta_{0S_0} \bbeta_{0S_0}\T )n^{-1}\Xbb_{,S_0}\T \bY\|_{\infty}.
\end{align}
To deal with the first term we make usage of the following result:

\begin{lemma} \label{inftyinftynorm} Suppose that $s,n$ satisfy $\frac{s}{n}\leq \frac{1}{16}$. The following bound holds:
$$\|[n^{-1}\Xbb_{,S_0}\T \Xbb_{,S_0}]^{-1} - [n^{-1} \widetilde \Xbb_{,S_0}\T \widetilde \Xbb_{,S_0}]^{-1}\|_{\infty, \infty} \leq 40\sqrt{s}\left(C_1 \frac{s}{n}\|\bbeta_{0S_0}\|_{\infty} + \tilde C_2 \sqrt{\frac{\log p}{n}}\right),$$
with probability at least $1 - 4 \exp(-s/2) - \frac{12}{p} - \frac{4}{n}$, where $C_1 > 0$ and $\tilde C_2 = C_2 + 4$ are the same constants as in (\ref{i1:i4:bound}).
\end{lemma} 
%\begin{remark} The constant $C_1 = \sqrt{\frac{256 s \log(p-s)}{C_3 n}}$ with $C_3 > 0$ being an absolute constant, can be chosen to be arbitrary small, for the sake of making $n$ proportionally large comparable to $s \log(p-s)$. 
%\end{remark}

Note that the second term is a mean 0 random variable since $(\eye - \bbeta_{0S_0} \bbeta_{0S_0}\T) \Xbb_{,S_0}$ is independent of $\bY$. Just as in Lemma \ref{subgaussboundnormal} we can show that $\|(\eye - \bbeta_{0S_0} \bbeta_{0S_0}\T )n^{-1}\Xbb_{,S_0}\T \bY\|_{\infty} \leq 2\sqrt{2 \sigma^2 \frac{\log p}{n}}$ with probability at least $1 - \frac{2s}{p^2} \geq 1 - \frac{2}{p}$ (this event is in fact a sub-event of the bounds of the first term $n^{-1} \Xbb_{,S_0}\T \bY - c_0 \bbeta_{0S_0}$). Lemma \ref{inftyinftynorm} gives us a bound on $\|n([\Xbb_{,S_0}\T \Xbb_{,S_0}]^{-1} - [\widetilde \Xbb_{,S_0}\T \widetilde \Xbb_{,S_0}]^{-1})\|_{\infty,\infty}$ which in conjunction with the previous inequality suffices to control the term $\bI_2$. %Note here that by $O(1)$ we mean that this term can be bounded by arbitrarily small constant with high probability at the expense for making the ratio $\frac{n}{s \log(p-s)}$ high.

Finally, to deal with the term $\bI_4$ we will make use of the following:
\begin{lemma} \label{sphereconclemma} Let $\frac{s}{n} \leq \frac{1}{64}$. Then there exists a constant $\Upsilon \asymp \sigma > 0$, such that the term:
$$
\|(n[\widetilde \Xbb_{,S_0}\T \widetilde \Xbb_{,S_0}]^{-1} - \eye)n^{-1} \widetilde \Xbb_{,S_0}\T \bY\|_{\infty} \leq \Upsilon \sqrt{\frac{\log p}{n}},
$$
with probability at least $1 - \frac{2}{p} - \frac{\eta}{n \sigma^4} - 2\exp(-s/2)$.
\end{lemma}

Applying Lemma \ref{sphereconclemma} we have in conjunction with our previous bounds (\ref{i1:i4:bound}) and (\ref{i2:bound}) we get:
\begin{align*}
\|[\Xbb_{,S_0}\T \Xbb_{,S_0}]^{-1} \Xbb_{,S_0}\T \bY - c_0 \bbeta_{0S_0}\|_{\infty} & \leq \Big(C_1\frac{s}{n}\|\bbeta_{0S_0}\|_{\infty} + C_2 \sqrt{\frac{\log p}{n}}\Big)\left(|c_0| + \sqrt{\frac{\log p}{n}} +  2 \sqrt{\sigma^2  \frac{\log p}{n}}\right)\\
& + 80\sqrt{s}\left(C_1 \frac{s}{n}\|\bbeta_{0S_0}\|_{\infty} + \tilde C_2 \sqrt{\frac{\log p}{n}}\right)\sqrt{2 \sigma^2 \frac{\log p}{n}}\\
& + \Upsilon \sqrt{\frac{\log p}{n}} +  \|\bbeta_{0S_0}\|_{\infty}\sqrt{\frac{\log p}{n}}+  2\sqrt{\sigma^2 \frac{\log p}{n}},
\end{align*}
with probability at least $1 - 4 \exp(-s/2) - \frac{4}{n} - \frac{18}{p} -2\frac{\eta}{n \sigma^4} - 2\frac{\gamma}{\log p}$\footnote{Here we are recognizing the fact that the events of some probability bounds we derived above, in fact coincide.}, which finishes the proof, after grouping terms and recalling the fact that $\log(p-s) \asymp \log p$.
\end{proof}

\begin{proof}[Proof of Lemma \ref{inftyinftynorm}] We first compare $[n^{-1}\Xbb_{,S_0}\T \Xbb_{,S_0}]^{-1}$ to $[n^{-1} \widetilde \Xbb_{,S_0}\T \Xbb_{,S_0} + \bbeta_{0S_0} \bbeta_{0S_0}\T]^{-1}$. The latter matrix might happen to be non-invertible but this is irrelevant for our proof as we argue below. %We start by arguing that the latter matrix is indeed invertible wp 1. %Note that 
%$$n^{-1} \widetilde \Xbb_{,S_0}\T \Xbb_{,S_0} + \bbeta_{0S_0} \bbeta_{0S_0}\T= (\eye -\bbeta_{0S_0} \bbeta_{0S_0}\T)n^{-1} \Xbb_{,S_0}\T \Xbb_{,S_0} + \bbeta_{0S_0} \bbeta_{0S_0}\T(\eye + n^{-1}\Xbb_{,S_0}^{*\intercal}\Xbb_{,S_0}),$$ 
%and since the matrix $\Xbb_{,S_0}\T \Xbb_{,S_0}$ is full rank wp 1 and $\bbeta_{0S_0}\T(\eye + n^{-1}\Xbb_{,S_0}^{*\intercal}\Xbb_{,S_0}) \neq 0$ wp 1, the matrix in question is of full column rank. 
Using Woodbury's matrix identity we have:
$$
[n^{-1} \widetilde \Xbb_{,S_0}\T \Xbb_{,S_0} + \bbeta_{0S_0} \bbeta_{0S_0}\T]^{-1} - [n^{-1}\Xbb_{,S_0}\T \Xbb_{,S_0}]^{-1} = \frac{ [n^{-1}\Xbb_{,S_0}\T \Xbb_{,S_0}]^{-1} \bbeta_{0S_0} \bbeta_{0S_0}\T \Mb [n^{-1}\Xbb_{,S_0}\T \Xbb_{,S_0}]^{-1}}{1 - \bbeta_{0S_0}\T \Mb [n^{-1}\Xbb_{,S_0}\T \Xbb_{,S_0}]^{-1} \bbeta_{0S_0}},
$$
where $\Mb = n^{-1}\Xbb_{,S_0}\T \Xbb_{,S_0} - \eye - n^{-1} \Xbb_{,S_0}^{*\intercal}\Xbb_{,S_0}$.  Note that whenever the right hand side of Woodbury's identity is well defined, the matrix $n^{-1} \widetilde \Xbb_{,S_0}\T \Xbb_{,S_0} + \bbeta_{0S_0} \bbeta_{0S_0}\T$ is indeed invertible, and the inverse satisfies the above identity. As we argue below the right hand side is well defined (i.e. the denominator is non-zero) with high probability hence the proof goes through. 
Next we handle the term $\bbeta_{0S_0}\T \Mb [n^{-1}\Xbb_{,S_0}\T \Xbb_{,S_0}]^{-1}$. By the triangle inequality have:
%To deal with the term $[n^{-1}X\T X]^{-1} \bbeta$, note that we are in a position to apply Lemma 5 from Wainwright's paper. Applying this lemma gives us the existence of constants $C_1$ and $C_2$ such that $\|[n^{-1}X\T X]^{-1} \bbeta\|_{\infty} \leq (1 + C_1) \|\bbeta\|_{\infty} $ with probability at least $1 - 4\exp(-C_2 \min(k, \log (p-k)))$. Ne`xt we handle the term $\bbeta\T M [n^{-1}X\T X]^{-1}$. By the triangle inequality have:
$$
\|\bbeta_{0S_0}\T \Mb [n^{-1}\Xbb_{,S_0}\T \Xbb_{,S_0}]^{-1}\|_{\infty} \leq \|\bbeta_{0S_0}\T ([n^{-1}\Xbb_{,S_0}\T \Xbb_{,S_0}]^{-1} - \eye)\|_{\infty} + \|\bbeta_{0S_0}\T n^{-1} \Xbb_{,S_0}^{*\intercal}\Xbb_{,S_0} [n^{-1}\Xbb_{,S_0}\T \Xbb_{,S_0}]^{-1}\|_{\infty}.
$$
For the first term Lemma \ref{lemma5wain:mod} is directly applicable. Applying this lemma gives us the existence of constants $C_1$ and $C_2$ such that:
$$\|([n^{-1}\Xbb_{,S_0}\T \Xbb_{,S_0}]^{-1} - \eye) \bbeta_{0S_0}\|_{\infty} \leq C_1 \frac{s}{n}\|\bbeta_{0S_0}\|_{\infty} + C_2 \sqrt{\frac{\log p}{n}},$$
 with probability at least $1 - 4p^{-1}$. For the second term, we have that conditionally on $\Xbb_{,S_0}$ it has a normal distribution: $\cN(0, n^{-1} (n^{-1}\Xbb_{,S_0}\T \Xbb_{,S_0})^{-1})$. Since $\Xbb_{,S_0}$ is standard normal, we can apply Lemma \ref{cor535versh} to claim that $\|n(\Xbb_{,S_0}\T \Xbb_{,S_0})^{-1}\|_{2,2} \leq \left(\frac{1}{1 - \sqrt{\frac{s}{n}} - t}\right)^2$ with probability at least $1 - 2 \exp(-nt^2/2)$. Taking $t = \sqrt{\frac{s}{n}}$ gives us that $\|n(\Xbb_{,S_0}\T \Xbb_{,S_0})^{-1}\|_{2,2} \leq \frac{1}{(1 - 2\sqrt{\frac{s}{n}})^2}$ with probability at least $1 - 2 \exp(-s/2)$. Thus conditioning on this event, by a standard normal tail bound and a union bound we have:
$$
\p(\|\bbeta_{0S_0}\T n^{-1} \Xbb_{,S_0}^{*\intercal}\Xbb_{,S_0} [n^{-1}\Xbb_{,S_0}\T \Xbb_{,S_0}]^{-1}\|_{\infty} \geq t) \leq 2 s \exp\left(- t^2n\left(1 - 2\sqrt{\frac{s}{n}}\right)^2/{2}\right).
$$
Selecting $t = 4\sqrt{\frac{\log p}{n}}$, we get the probability above is bounded by $\frac{2s}{p^2} \leq \frac{2}{p}$ (where we used the assumption $\sqrt{\frac{s}{n}} \leq \frac{1}{4}$). So finally on the intersection event we have:
$$
\|\bbeta_{0S_0}\T \Mb [n^{-1}\Xbb_{,S_0}\T \Xbb_{,S_0}]^{-1}\|_{\infty} \leq C_1 \frac{s}{n}\|\bbeta_{0S_0}\|_{\infty} + \tilde C_2 \sqrt{\frac{\log p}{n}},
$$
with probability at least $1 - 6p^{-1} - 2\exp(-s/2)$ where $\tilde C_2 = C_2 + 4$. Let us now consider the denominator:
\begin{align*}
& 1 - \bbeta_{0S_0}\T \Mb [n^{-1}\Xbb_{,S_0}\T \Xbb_{,S_0}]^{-1} \bbeta_{0S_0}\\
= & 1 -  \bbeta_{0S_0}\T( \eye -  [n^{-1}\Xbb_{,S_0}\T \Xbb_{,S_0}]^{-1}) \bbeta_{0S_0} + n^{-1}\bbeta_{0S_0}\T \Xbb_{,S_0}^{*\intercal} \Xbb_{,S_0} [n^{-1}\Xbb_{,S_0}\T \Xbb_{,S_0}]^{-1} \bbeta_{0S_0}\\
= & \bbeta_{0S_0}\T [n^{-1}\Xbb_{,S_0}\T \Xbb_{,S_0}]^{-1}\bbeta_{0S_0} + n^{-1}\bbeta_{0S_0}\T \Xbb_{,S_0}^{*\intercal} \Xbb_{,S_0} [n^{-1}\Xbb_{,S_0}\T \Xbb_{,S_0}]^{-1} \bbeta_{0S_0}.
\end{align*}
%Using Lemma \ref{cor535versh} we have $\|\eye - [n^{-1}\Xbb_{,S_0}\T \Xbb]^{-1}\|_2  \leq \frac{1}{(1 - 2\sqrt{\frac{s}{n}})^2} - 1 \leq \frac{7}{9}$ with the last bound holding under the condition $\frac{s}{n} \leq \frac{1}{64}$. Hence $|\bbeta\T(\eye -  [n^{-1}\Xbb\T \Xbb]^{-1}) \bbeta| \leq \frac{7}{9}$. For the second term just as before, conditionally on $\bfX$ we have $n^{-1}\bbeta\T \bfX^{*\intercal} \bfX [n^{-1}\Xbb\T \Xbb]^{-1} \bbeta \sim \cN(0, n^{-1}\bbeta\T [n^{-1}\Xbb\T \Xbb]^{-1}\bbeta)$. Then by a standard tail bound we have that the second term is $\leq 4 \sqrt{\frac{\log n}{n}}$ with probability at least $1 - \frac{2}{n}$. Putting everything together we have:
Using Lemma \ref{cor535versh} we have $\lambda_{\min}([n^{-1}\Xbb_{,S_0}\T \Xbb_{,S_0}]^{-1})  \geq \frac{1}{(1 + 2\sqrt{\frac{s}{n}})^2} > \frac{1}{4}$ with the last bound holding since $\frac{s}{n} < \frac{1}{4}$. Hence $\bbeta_{0S_0}\T [n^{-1}\Xbb_{,S_0}\T \Xbb_{,S_0}]^{-1}\bbeta_{0S_0} \geq \frac{1}{4}$. For the second term just as before, conditionally on $\Xbb_{,S_0}$ we have 
$$n^{-1}\bbeta_{0S_0}\T \Xbb_{,S_0}^{*\intercal} \Xbb_{,S_0} [n^{-1}\Xbb_{,S_0}\T \Xbb_{,S_0}]^{-1} \bbeta_{0S_0} \sim \cN(0, n^{-1}\bbeta_{0S_0}\T [n^{-1}\Xbb_{,S_0}\T \Xbb_{,S_0}]^{-1}\bbeta_{0S_0}).$$ 
Then (given that $\|n(\Xbb_{,S_0}\T \Xbb_{,S_0})^{-1}\|_{2,2} \leq \frac{1}{(1 - 2\sqrt{\frac{s}{n}})^2}$)  by a standard tail bound we have that the second term is $\leq 4 \sqrt{\frac{\log n}{n}}$ with probability at least $1 - \frac{2}{n}$. Putting everything together we have:
$$
1 - \bbeta_{0S_0}\T \Mb [n^{-1}\Xbb_{,S_0}\T \Xbb_{,S_0}]^{-1} \bbeta_{0S_0} \geq \frac{1}{4} - 4 \sqrt{\frac{\log n}{n}}.
$$
The last expression is clearly bigger than $\frac{1}{5}$ for large enough values of $n$. Hence we conclude that with high probability we have:
\begin{align*}
& \|[n^{-1} \widetilde \Xbb_{,S_0}\T \Xbb_{,S_0} + \bbeta_{0S_0} \bbeta_{0S_0}\T]^{-1} - [n^{-1}\Xbb_{,S_0}\T \Xbb_{,S_0}]^{-1}\|_{\infty,\infty} \\
& \leq 5 \|[n^{-1}\Xbb_{,S_0}\T \Xbb_{,S_0}]^{-1} \bbeta_{0S_0}\|_{1}\|\bbeta_{0S_0}\T \Mb [n^{-1}\Xbb_{,S_0}\T \Xbb_{,S_0}]^{-1}\|_{\infty} 
\end{align*}
For the first term, by the definition of matrix $\|\cdot\|_{2,2}$ norm we further have:
\begin{align*}
\|[n^{-1}\Xbb_{,S_0}\T \Xbb_{,S_0}]^{-1} \bbeta_{0S_0}\|_{1} & \leq \sqrt{s} \|[n^{-1}\Xbb_{,S_0}\T \Xbb_{,S_0}]^{-1} \bbeta_{0S_0}\|_{2} \leq  \sqrt{s} \|\bbeta_{0S_0}\|_2 \|[n^{-1}\Xbb_{,S_0}\T \Xbb_{,S_0}]^{-1}\|_{2,2}\\
& \leq \frac{\sqrt{s}}{(1 - 2\sqrt{\frac{s}{n}})^2}.
\end{align*}
Combining this inequality with our previous bound we get:
$$
\|[n^{-1} \widetilde \Xbb_{,S_0}\T \Xbb_{,S_0} + \bbeta_{0S_0} \bbeta_{0S_0}\T]^{-1} - [n^{-1}\Xbb_{,S_0}\T \Xbb_{,S_0}]^{-1}\|_{\infty,\infty} \leq  \frac{5 \sqrt{s}}{(1 - 2\sqrt{\frac{s}{n}})^2}\left(C_1 \frac{s}{n}\|\bbeta_{0S_0}\|_{\infty} + \tilde C_2 \sqrt{\frac{\log p}{n}}\right).
$$

Next we show that $[n^{-1} \widetilde \Xbb_{,S_0}\T \Xbb_{,S_0} + \bbeta_{0S_0} \bbeta_{0S_0}\T]^{-1}$ is also close to $[n^{-1}\widetilde \Xbb_{,S_0}\T\widetilde \Xbb_{,S_0}]^{-1}$. Another usage of Woodbury's matrix identity yields:
$$
[n^{-1} \widetilde \Xbb_{,S_0}\T \Xbb_{,S_0} + \bbeta_{0S_0} \bbeta_{0S_0}\T]^{-1} - [n^{-1}\widetilde \Xbb_{,S_0}\T \widetilde \Xbb_{,S_0}]^{-1} = \frac{ [n^{-1}\widetilde \Xbb_{,S_0}\T \widetilde \Xbb_{,S_0}]^{-1} \widetilde \Mb \bbeta_{0S_0} \bbeta_{0S_0}\T [n^{-1}\widetilde \Xbb_{,S_0}\T \widetilde \Xbb_{,S_0}]^{-1}}{1 - \bbeta_{0S_0}\T  [n^{-1}\widetilde \Xbb_{,S_0}\T \widetilde \Xbb_{,S_0}]^{-1} \widetilde \Mb\bbeta_{0S_0}},
$$
where $\widetilde \Mb = n^{-1}\widetilde \Xbb_{,S_0}\T \widetilde \Xbb_{,S_0} - \eye - n^{-1}\widetilde \Xbb_{,S_0}\T \Xbb_{,S_0}$. Note that since $\widetilde \Xbb_{,S_0}\T \independent \Xbb_{,S_0} \bbeta_{0S_0} $, the same argument as before goes through. Combining the bounds with a triangle inequality completes the proof, using the fact that $\sqrt{\frac{s}{n}} \leq \frac{1}{4}$.
\end{proof}

\begin{proof}[Proof of Lemma \ref{sphereconclemma}] We first perform a singular value decomposition on the $\widetilde \Xbb_{,S_0} = \Ub_{n \times s} \Db_{s \times s} \Vb_{s \times s}\T$ matrix. Note that since multiplying $\widetilde \Xbb_{,S_0}$ by a unitary $s \times s$ matrix on the right, or by a unitary $n \times n$ matrix on the left doesn't change the distribution of $\widetilde \Xbb_{,S_0}$ we conclude that the matrices $\Ub, \Db$ and $\Vb$ are independent. This representation gives us that $(n^{-1}\widetilde \Xbb_{,S_0}\T \widetilde \Xbb_{,S_0})^{-1} - \eye = \Vb (n\Db^{-2} - \eye) \Vb\T$. With this notation we can rewrite:
$$
(n[\widetilde \Xbb_{,S_0}\T \widetilde \Xbb_{,S_0}]^{-1} - \eye)n^{-1} \widetilde \Xbb_{,S_0}\T \bY = \Vb \underbrace{(n\Db^{-2} - \eye) n^{-1/2} \Db}_{\Wb} n^{-1/2} \Ub\T \bY.
$$
We recall that by construction $\widetilde \Xbb_{,S_0}$ is independent of $\bY$. The elements of the matrix $\Wb$ can be bounded in a simple manner. We have $\|\Wb\|_{2,2} \leq \|(n \Db^{-2} - \eye)\|_{2,2} \|n^{-1/2} \Db\|_{2,2}$, and by Lemma \ref{cor535versh}, as before we have: $\|(n \Db^{-2} - \eye)\|_{2,2} \leq  \frac{1}{(1 - 2\sqrt{\frac{s}{n}})^2} - 1 \leq \frac{4\sqrt{\frac{s}{n}}}{(1 - 2\sqrt{\frac{s}{n}})^2}$ and $\|n^{-1/2}\Db\|_{2,2} \leq 1 + 2\sqrt{\frac{s}{n}}$ with probability at least $1 - 2 \exp(-s/2)$. We will condition on the event $\|\Wb\|_{2,2} \leq \frac{4\sqrt{\frac{s}{n}}}{(1 - 2\sqrt{\frac{s}{n}})^2}(1 + 2\sqrt{\frac{s}{n}}) < 9\sqrt{\frac{s}{n}}$, with the last inequality holding for $\sqrt{\frac{s}{n}} \leq \frac{1}{8}$. Since every random variable in the above display is independent from $\Wb$, the distributions of $\Vb, \Ub$ and $\bY$ stay unchanged under this conditioning. Let $\eb_i$ be a unit vector with $1$ on the $i$\textsuperscript{th} position. Since we are interested in bounding the $\|\cdot\|_{\infty}$ we will start with the following:
$$
\eb_i\T (n[\widetilde \Xbb_{,S_0}\T \widetilde \Xbb_{,S_0}]^{-1} - \eye)n^{-1} \widetilde \Xbb_{,S_0}\T \bY = \vb_i\T \Wb[n^{-1/2} \Ub\T \bY],
$$
where $\vb_i\T$ is the $i$\textsuperscript{th} row of the matrix $\Vb$. Condition on the vector $n^{-1/2} \Ub\T \bY$. Since $\vb_i$ is independent of $n^{-1/2} \Ub\T \bY$ it follows that the distribution of $\vb_i$ is uniform on the unit sphere in $\mathbb{R}^s$. We next show that the function $F(\vb_i) = \vb_i\T \Wb[n^{-1/2} \Ub\T \bY]$ is Lipschitz. We have:
\begin{align*}
\|\nabla F\|_2 & \leq \|\Wb\|_{2,2} \|n^{-1/2} \Ub\T \bY\|_2 \leq 9\sqrt{\frac{s}{n}} n^{-1/2} \sqrt{\sum_{i = 1}^s (\ub_i\T \bY)^2}\\
& \leq 9\sqrt{\frac{s}{n}} n^{-1/2} \|\bY\|_2,
\end{align*}
where the last inequality follows from the fact that the vectors $\ub_i$ are orthonormal and hence $\sum_{i = 1}^s (\ub_i\T \bY)^2 \leq \|\bY\|^2_2$. Since $Y_i$ are assumed to have finite second moment, by Chebyshev's inequality we have that:
$$
\p(|n^{-1} \|\bY\|^2_2 - \sigma^2| \geq t) \leq \frac{\eta}{n t^2}.
$$
Selecting $t = \sigma^2$ is sufficient to keep the above probability going to 0, and furthermore for $n$ large enough guarantees that $n^{-1} \|\bY\|^2_2 \leq 2 \sigma^2 $ and hence $n^{-1/2}\|\bY\|_2 \leq \sqrt{2} \sigma$. Thus conditional on this event the function $F$ is Lipschitz with a constant equal to $\sqrt{2} 9\sigma \sqrt{\frac{s}{n}}$. Since the expectation of the function $F$ is 0, by concentration of measure for Lipschitz functions on the sphere \citep{ledoux2005concentration, ledoux2013probability}, for any $t > 0$ we have:
$$
\p(|F(\vb_i)| \geq t \sigma) \leq 2\exp\left(- \tilde c s \frac{t^2}{162 \frac{s}{n}}\right),
$$
for some absolute constant $\tilde c > 0$. Taking a union bound the above becomes:
$$
\p(\max_{i \in [s]}|F(\vb_i)| \geq t \sigma) \leq 2 s\exp\left( - \tilde c  \frac{t^2n}{162}\right).
$$
Selecting $t = 18 \sqrt{\frac{\log p}{\tilde c n}}$, keeps the probability vanishing at a rate faster than $2s/p^2 \leq 2/p$ and completes the proof.
\end{proof}

\begin{proof}[Proof of Corollary \ref{cor:mainres:transform}] Tracing the proof of Theorem \ref{mainresult} we realize that it suffices  to show the following two quantities remain well controlled under the usage of $\hat g$:
\begin{itemize}
\item[i.] $|n^{-1}\sum_{i = 1}^n\bX_i\T\bbeta_0 \hat g(Y_i) - c_0| \leq O(\sqrt{\log p/n})$,
\item[ii.] $n^{-1} \sum_{i = 1}^n \hat g^2(Y_i) = O(1)$,
\end{itemize}
with probability at least $1 - O(p^{-1})$ and $1 - O(n^{-1})$ correspondingly.
To deal with i. observe that:
\begin{align*}
\left| n^{-1}\sum_{i = 1}^n\bX_i\T\bbeta_0 \hat g(Y_i) - c_0\right| & \leq \left|n^{-1}\sum_{i = 1}^n\bX_i\T\bbeta_0 g(Y_i) - c_0\right| + \left|n^{-1}\sum_{i = 1}^n\bX_i\T\bbeta_0 (g(Y_i) - \hat g(Y_i))\right| \\
& \leq \underbrace{\left|n^{-1}\sum_{i = 1}^n\bX_i\T\bbeta_0 g(Y_i) - c_0\right|}_{I_1} \\
& + \underbrace{\Big\{n^{-1}\sum_{i = 1}^n (\bX_i\T\bbeta_0)^2\Big\}^{1/2}\Big\{n^{-1}\sum_{i = 1}^n (\hat g(Y_i) - g(Y_i))^2\Big\}^{1/2}}_{I_2}.
\end{align*}
The term $I_1$ remains controlled by the proof of Theorem \ref{mainresult}, while for the term $I_2$ we have:
$$
I_2 \leq O(\sqrt{\log p/n}), 
$$
with probability at least $1 - O(p^{-1})$, where we used the assumption on $\hat g$ and the fact that the random variables $(\bX_i\T\bbeta_0)^2 \sim \chi^2_{1}$ and hence concentrate exponentially about their mean --- $1$, by a standard tail bound \citep{boucheron2013concentration}. 

Next, for ii., by the triangle inequality we have:
$$
\sqrt{n^{-1} \sum_{i = 1}^n \hat g^2(Y_i)} \leq \sqrt{n^{-1} \sum_{i = 1}^n g^2(Y_i)} + \sqrt{n^{-1} \sum_{i = 1}^n (\hat g(Y_i) - g(Y_i))^2}.
$$
The first term is well controlled as before and is $O(1)$ with probability at least $1 - O(n^{-1})$ and the second term is at most $O(\sqrt{\log p/n})$ with probability at least $1 - O(p^{-1})$ by assumption which concludes the proof.
\end{proof}

\begin{proof}[Proof of Proposition \ref{rao:blackwell:type:of:res}] First let $g$ be such that $\EE\{g(Y)\bX\T\bbeta_0\} \neq 0$. Recall that $\EE(\bX\T\bbeta_0) = 0$. Hence by Cauchy-Schwartz we have:
$$
0 < [\EE\{g(Y)\bX\T\bbeta_0\}]^2 = (\EE [g(Y) \EE\{\bX\T\bbeta_0|Y\}])^2 \leq \Var\{g(Y)\} \Var\{\EE(\bX\T\bbeta_0 | Y)\} ,
$$
and therefore $\Var\{\EE(\bX\T\bbeta_0 | Y)\}  > 0$. %Let $W = \EE(\bX\T\bbeta_0|Y)$. By the conditional variance properties we have:
%\begin{align*}
%\Var(W) & = \EE[\var\{W | g(Y)\}] + \var[\EE\{W | g(Y)\}] \geq  \var[\EE\{W | g(Y)\}]\\
%& = \var(\EE[\bX\T\bbeta_0|Y]) > 0,
%\end{align*}
%where the last equality holds by the tower property of conditional expectation. 
In the reverse case put $g(Y) = \EE\{\bX\T\bbeta_0 | Y\}$ and apply conditional expectation to obtain $\EE\{g(Y)\bX\T\bbeta_0\} = \Var\{\EE(\bX\T\bbeta_0 | Y)\} > 0$.
\end{proof}

%\section{Modifications TO MERGE WITH ACTUAL FILE}

\section{Additional Simulation Results}\label{more:plots}

\subsection[]{$\bSigma = \eye$}

\begin{figure}[H]
  \centering
  \includegraphics[width=1\linewidth]{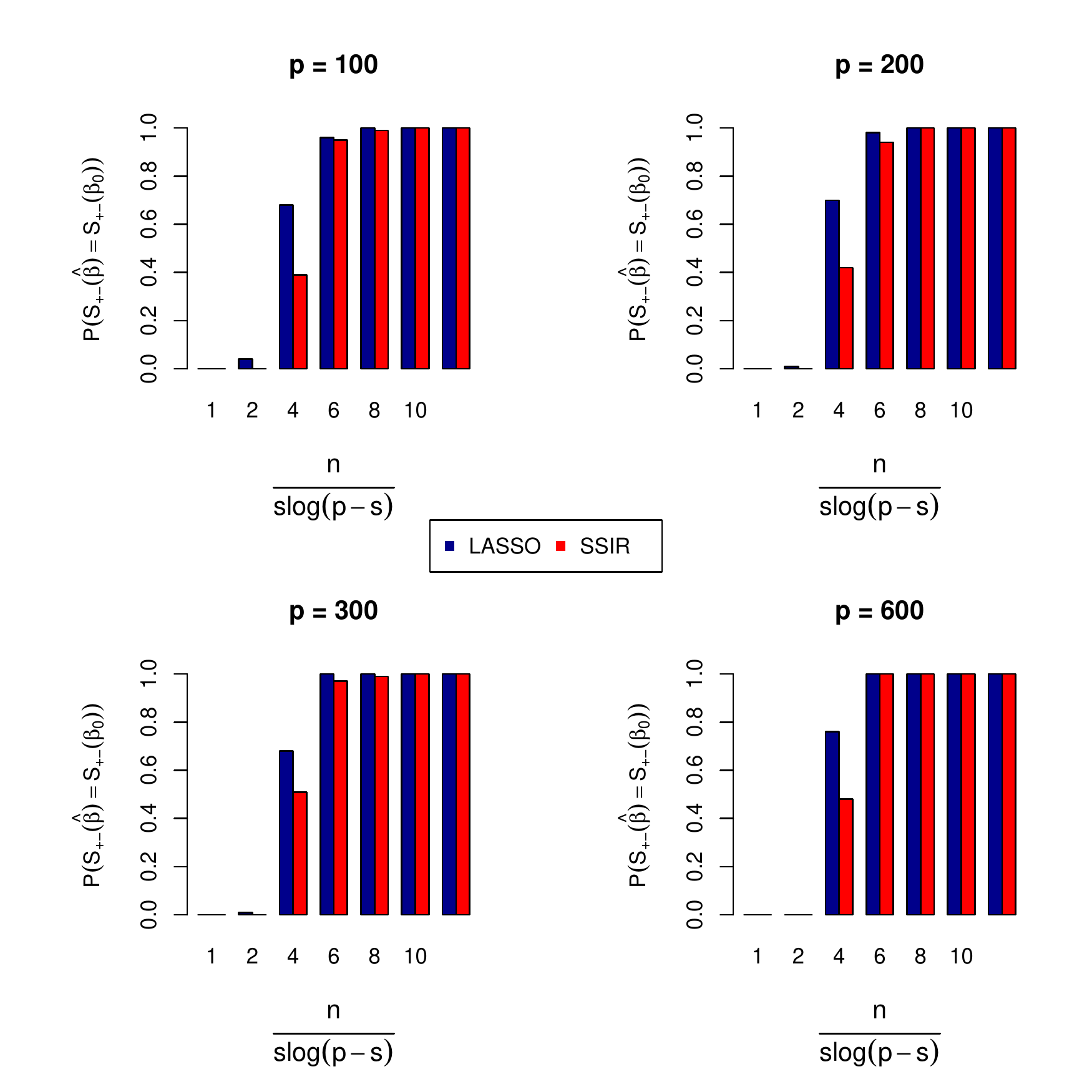}
  \caption{Model (\ref{sinmodel})}
\end{figure}%

\begin{figure}[H]
  \centering
  \includegraphics[width=1\linewidth]{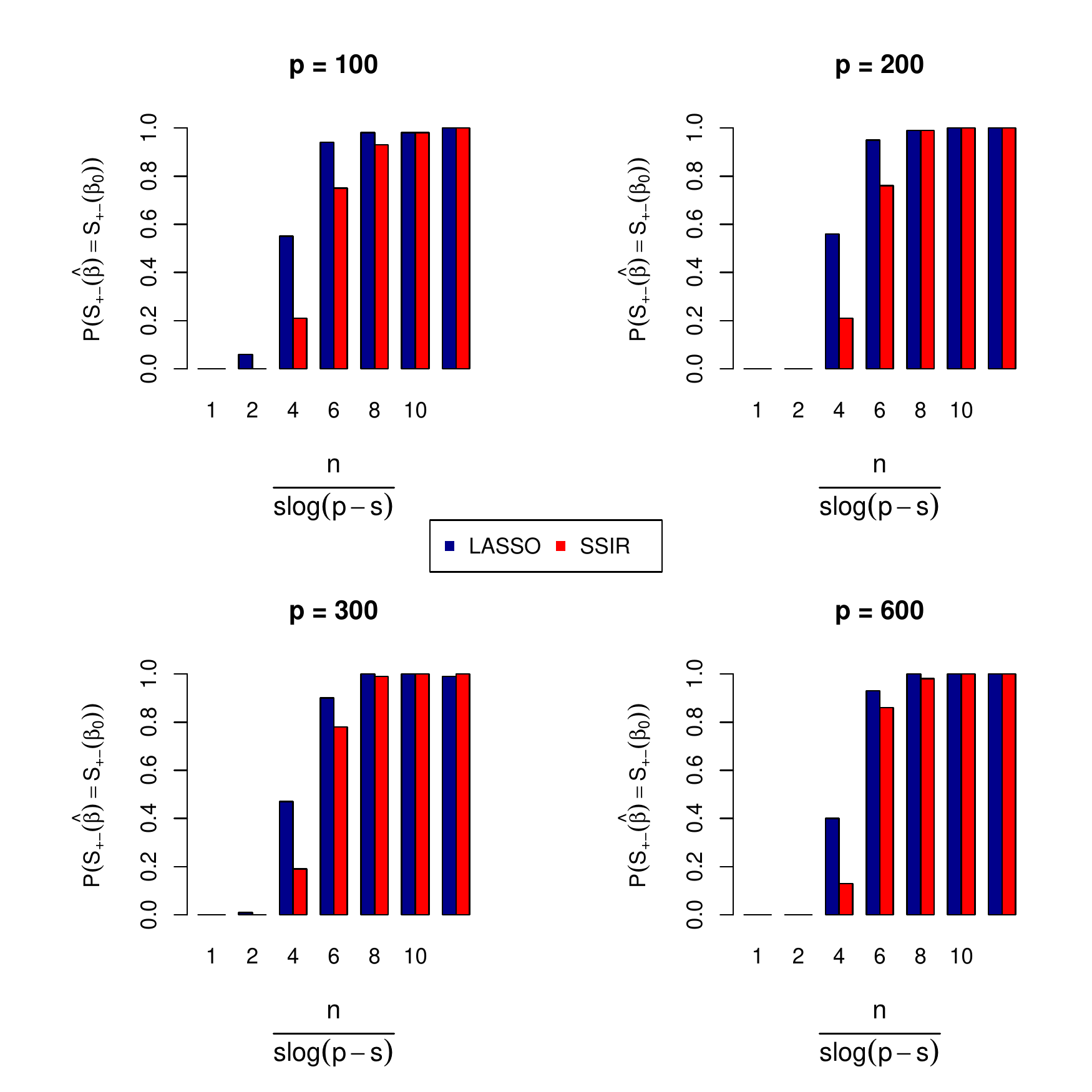}
  \caption{Model (\ref{regmodel})}
\end{figure}%

\subsection[]{$\bSigma: \Sigma_{kj} = 2^{-|k-j|}$}

\begin{figure}[H]
  \centering
  \includegraphics[width=1\linewidth]{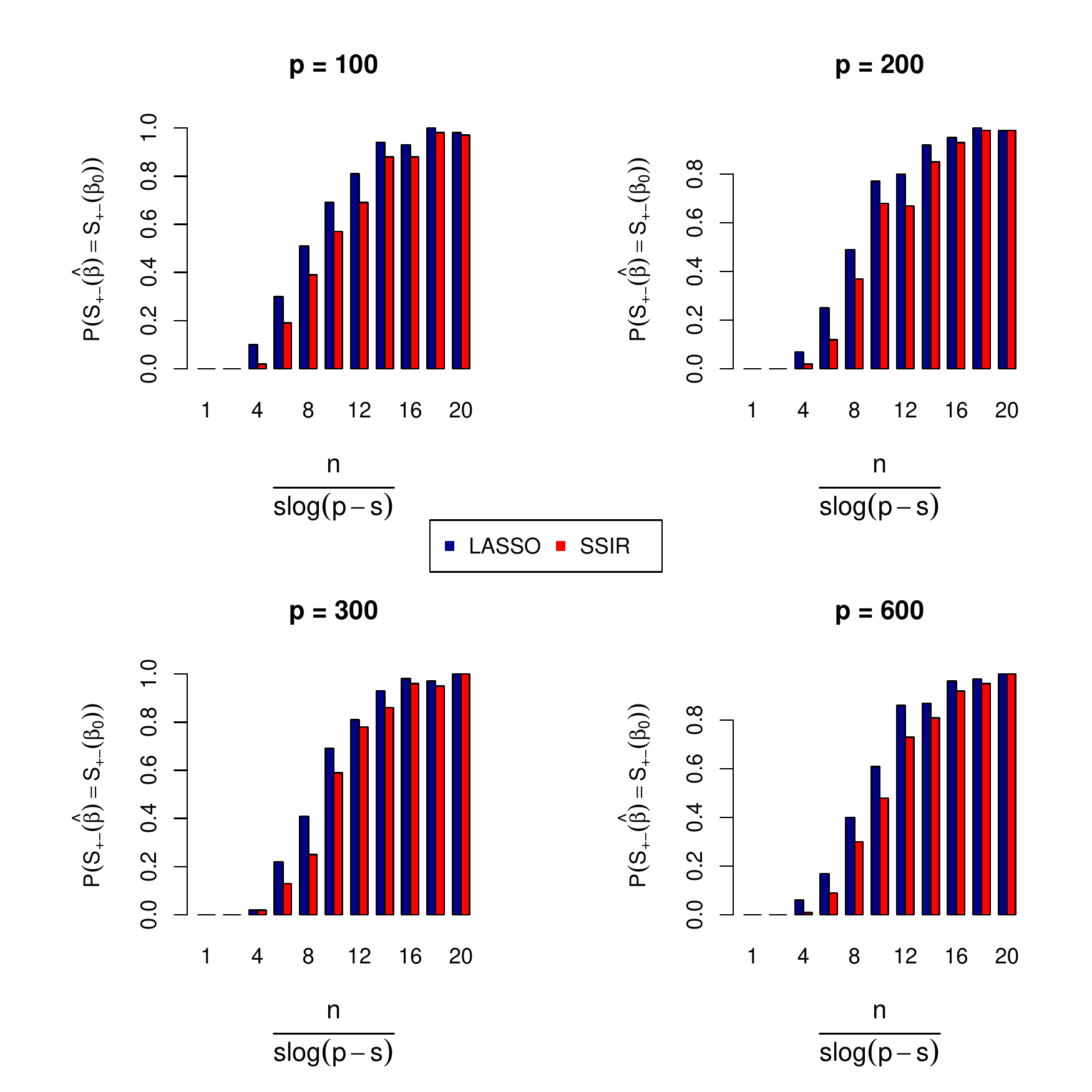}
  \caption{Model (\ref{sinmodel})}
\end{figure}%

\begin{figure}[H]
  \centering
  \includegraphics[width=1\linewidth]{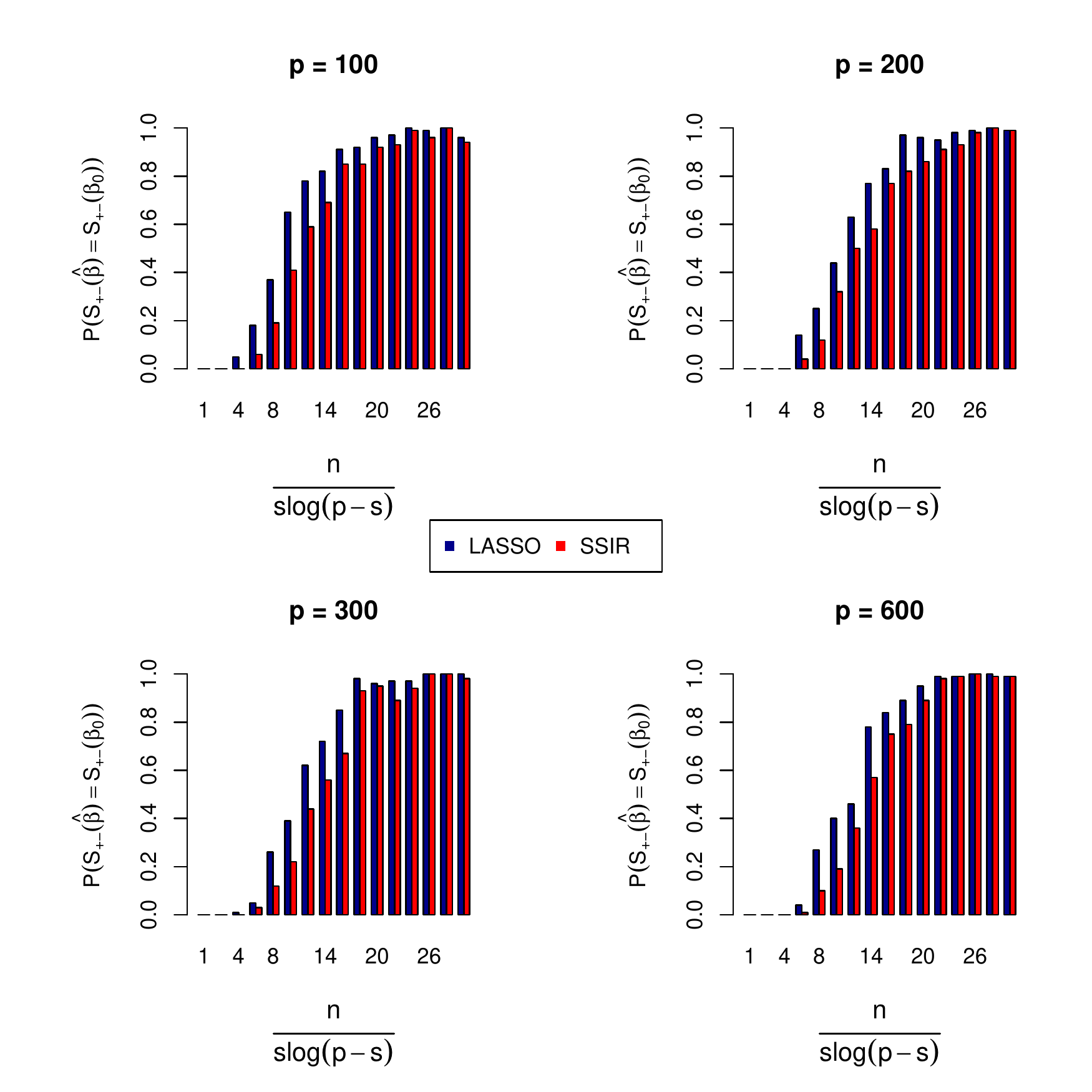}
  \caption{Model (\ref{regmodel})}
\end{figure}%

\newpage

\bibliographystyle{agsm}
\bibliography{some_references}

\end{document}